\documentclass[12pt]{amsart}
\usepackage{amsthm,amsopn,amsfonts,amssymb,pb-diagram}
\usepackage{hyperref}
\usepackage{fancyhdr}
\usepackage[dvips]{graphics}

\usepackage{euler, amsfonts, amssymb, latexsym, epsfig, epic}
\usepackage{epstopdf}

\usepackage[OT2,OT1]{fontenc}
\newcommand\cyr{%
  \renewcommand\rmdefault{wncyr}%
  \renewcommand\sfdefault{wncyss}%
  \renewcommand\encodingdefault{OT2}%
  \normalfont
  \selectfont}
\DeclareTextFontCommand{\textcyr}{\cyr}

\font\co=lcircle10
\def\boxcross{\ \smash{\lower6.5pt\hbox{\rlap{\hskip4.5pt\vrule height13.5pt}}
                \raise0pt\hbox{\rlap{\hskip-2pt \vrule height.4pt depth0pt
                        width13.5pt}}}\hskip12.7pt}

\def\boxelbow{\ \hskip.1pt\smash{%
               \hbox{\co \hskip 5.5pt\rlap{\mathsurround=0pt\rlap{\mathsurround=0pt\char'006}\lower0.4pt\rlap{\char'004}}
                \lower6.5pt\rlap{\hskip-0.2pt\vrule height3pt}
                \raise3.5pt\rlap{\hskip-0.2pt\vrule height3.2pt}}
                \hbox{%
                  \rlap{\hskip-6.4pt \vrule height.4pt depth0pt
width2.5pt}%
                  \rlap{\hskip4.05pt \vrule height.4pt depth0pt
width3.1pt}}}
                \hskip8.7pt}

\newtheorem{Theorem}{Theorem}[section]

\newtheorem*{Theorem*}{Theorem}
\newtheorem{Lemma}[Theorem]{Lemma}

\newtheorem{Proposition}[Theorem]{Proposition}

\newtheorem{Corollary}[Theorem]{Corollary}

\theoremstyle{definition}

\theoremstyle{remark}

\newcommand{\Gr}{\mathrm{Gr}}
\newcommand{\Grkn}{\Gr_k(\AA^n)}

\newcommand{\from}{\leftarrow}
\newcommand{\onto}{\twoheadrightarrow}
\newcommand{\fromonto}{\twoheadleftarrow}
\newcommand{\into}{\hookrightarrow}
\newcommand\mapsfrom{\reflectbox{$\mapsto$}}

\newcommand{\ZZ}{\mathbb{Z}}
\newcommand{\PP}{\mathbb{P}}

\newcommand{\fixit}[1]{\texttt{*** #1 ***}}

\newcommand\defn[1]{{\bf #1}}
\newcommand\yy{{\bf y}}
\newcommand\iso{\cong}
\newcommand\shift{\textcyr{X}}
\newcommand\sweep{\Psi}
\newcommand\St{{\rm St}}
\newcommand\dom{\,\backslash\,}
\newcommand\barX{{\overline X}}
\newcommand\integers{\ZZ}
\newcommand\codim{{\rm codim\space }}
\newcommand\calP{{\mathcal P}}
\newcommand\naturals{{\mathbb N}}
\newcommand\Kint{K\!\!\!\!\!\int\ }
\newcommand\calO{{\mathcal O}}
\newcommand\dash[1]{-\!\!\! #1\!\!\!-}
\newcommand\bslash{\big\backslash\!\!\!}
\newcommand\fslash{\big /\!\!\!}
\newcommand\union{\cup}

\renewcommand\AA{{\mathbb A}}
\newcommand\junk[1]{}

\begin{document}

\title{Puzzles, positroid varieties, and \\ equivariant $K$-theory
  of Grassmannians}
\author{Allen Knutson}
\address{Department of Mathematics, Cornell University, Ithaca, NY 14853 USA}
\email{allenk@math.cornell.edu}
\thanks{AK was partially supported by NSF grant DMS-0604708.}
\date{\today}

\begin{abstract}
  Vakil studied the intersection theory of Schubert varieties in the
  Grassmannian in a very direct way \cite{Vakil}: he degenerated the
  intersection of a Schubert variety $X_\mu$ and opposite Schubert
  variety $X^\nu$ to a union $\{X^\lambda\}$, with repetition.  
  This degeneration proceeds in stages, and along the way he met a
  collection of more complicated subvarieties, which he identified as
  the closures of certain locally closed sets.

  We show that Vakil's varieties are \emph{positroid varieties}, 
  which in particular shows they are normal, Cohen-Macaulay, have
  rational singularities, and are defined by the vanishing of
  Pl\"ucker coordinates \cite{KLS}. We determine the equations of the Vakil
  variety associated to a partially filled ``puzzle'' (building on the
  appendix to \cite{Vakil}), and extend Vakil's proof to give a geometric
  proof of the puzzle rule from \cite{KT} for equivariant Schubert calculus.

  The recent paper \cite{AGM} establishes (abstractly; without a formula) 
  three positivity results in equivariant $K$-theory of flag manifolds $G/P$.
\junk{
  We demonstrate two of these \fixit{I hope}, by modifying Vakil's
  degeneration of his subschemes to degenerations of certain ideal sheaves
  thereon, and giving a corresponding puzzle rule.
}
  We demonstrate one of these concretely, giving a corresponding puzzle rule.
\end{abstract}

\maketitle

{\footnotesize \tableofcontents}

\markright{\MakeUppercase{
    Puzzles, positroid varieties, and equivariant $K$-theory of Grassmannians}}

\section{Introduction, and statement of results}

\subsection{Schubert varieties and Vakil's geometric shifts}

Fix a Grassmannian $\Grkn$ of $k$-planes in affine $n$-space over a field.
One way to study it is as a quotient of the
\defn{Stiefel manifold} $\St_{k,n} \subseteq M_{k,n}$
of $k\times n$ matrices of full rank $k$; the map $\St_{k,n} \to \Grkn$
taking a matrix to its row span is surjective, and exactly mods out the
left action of $GL(k)$, which is by row operations.

We will use Greek letters $\lambda,\mu,\nu,\ldots$ to mean words 
of length $n$ with $n-k$ $0$s and $k$ $1$s. To each one, we associate
the varieties of matrices
$$ \barX_\lambda := \{ M \in M_{k,n} : \forall j=1,\ldots,n,\
        rank(M_{[1,j]}) \leq \#\text{$1$s in $\lambda$ at or before place $j$
        in $\lambda$} \} $$
$$ \barX^\mu := \{ M \in M_{k,n} : \forall i=1,\ldots,n,\
        rank(M_{[i,n]}) \leq \#\text{$1$s in $\mu$ at or after place $i$
        in $\mu$} \} $$
where $M_{[i,j]}$ indicates the $k\times (j-i+1)$ submatrix using
columns $i,i+1,\ldots,j$ of the $k\times n$ matrix $M$. Then
$$ X_\lambda := GL(k) \dom (\barX_\lambda \cap \St_{k,n}),\quad
   X^\mu := GL(k) \dom (\barX^\mu \cap \St_{k,n}) $$
are \defn{Schubert} and \defn{opposite Schubert varieties} in $\Grkn$,
and have
$$ \codim\ X_\lambda = \dim X^\lambda 
= |\lambda| := \# \big\{(i,j)\ :\ i<j,\ \lambda_i > \lambda_j\big\}. $$
These are well-known to be reduced and irreducible, and the set of 
Schubert varieties gives a $\integers$-basis of the cohomology ring.
Moreover, a cohomology class is effective if and only if it is a nonnegative
combination of Schubert classes. 

The coefficients $c_{\lambda\mu}^\nu$ in the multiplication
$[X_\lambda][X_\mu] = \sum c_{\lambda\mu}^\nu [X_\nu]$
arise in many contexts \cite{FultonSurvey}, and rules for computing
them are generically referred to as ``Littlewood-Richardson rules'',
though we will only apply this term to the Young-tableaux-based such rules.
In ``A geometric Littlewood-Richardson rule'' \cite{Vakil}, Vakil studies
this intersection problem in a very direct way; he degenerates by stages
the \defn{Richardson variety} $X_\lambda^\nu := X_\lambda \cap X^\nu$
to a union of opposite Schubert varieties, in which $X_\mu$ occurs
$c_{\lambda\mu}^\nu$ times.  (To avoid multiplicities cropping up in
his degenerate schemes, after each partial degeneration he must break
into components, and continue to degenerate them separately.)

Specifically, define the \defn{geometric shift} $\shift_{i \to j} X$
of a subscheme $X\subseteq \Grkn$ as the flat limit\footnote{%
  Consider pairs $\{ (t, \exp(t e_{ij}) \cdot x) : t\in \AA^1, x\in X)\}
  \subseteq \AA^1 \times \Grkn$, 
  and let $F$ be the closure in $\PP^1 \times \Grkn$.
  The flat limit is then the scheme-theoretic intersection 
  $F \cap \left( \{\infty\} \times \Grkn\right)$.
  The image of $F$ projected to $\Grkn$ is the ``geometric sweep''
  defined in that same paragraph.}
$\lim_{t\to\infty} \exp(t e_{ij}) \cdot X$, where $e_{ij}$ is a
matrix whose only nonzero entry is at $(i,j)$. (These are related to
the combinatorial shifts pioneered in \cite{EKR}, as we intend to
explain in a separate paper.) 
We define also a related operation on subvarieties
$X\subseteq \Grkn$, the \defn{geometric sweep} $\sweep_{i\to j} X$,
as the closure of $\bigcup_{t\in\AA^1} \exp(t e_{ij}) \cdot X$. 
So $\sweep_{i\to j} X$ is again irreducible,
and either $X = \shift_{i\to j} X = \sweep_{i\to j} X$,
or $\sweep_{i\to j} X$ contains $X$ and $\shift_{i\to j} X$
as rationally equivalent divisors.

We can describe already the principal geometric (rather than 
cohomological or combinatorial) results of this paper:

\begin{Theorem}\label{thm:geom}
  Recall Vakil's ``degeneration order'', 
  the following list of $n\choose 2$ pairs:
  $$ (n\!-\!1 \to n),\ $$
  $$ (n\!-\!2 \to n), (n\!-\!2 \to n\!-\!1),\  $$
  $$ (n\!-\!3 \to n), (n\!-\!3 \to n\!-\!1), (n\!-\!3 \to n\!-\!2),\ $$
  $$ \vdots $$
  $$ (n\to 1), (n\!-\!1\to 1), \qquad\qquad\ldots\qquad\qquad, 
  (4\to 1), (3\to 1), (2\to 1).$$
  Let $\shift_{\#i},\sweep_{\#i}$ denote the shift and sweep operations
  for the $i$th pair in this list.

  Let $X_0$ be a Richardson variety in $\Grkn$.
  As $i$ runs from $1$ to $n\choose 2$, apply $\shift_{\#i}$ to $X_{i-1}$, 
  then let $X_i$ be an irreducible component of $\shift_{\#i} X_{i-1}$.
  Vakil proves \cite[Theorem 5.10, Proposition 5.15]{Vakil} that
  regardless of these choices, each $\shift X_{i-1}$ is generically
  reduced, and has at most two components.  Also, $X_{n\choose 2}$ is
  an opposite Schubert variety.
  
  We generalize his process:
  as $i$ runs from $1$ to $n\choose 2$, apply either $\shift_{\#i}$ or
  $\sweep_{\#i}$ to $X_{i-1}$, then let $X_i$ be an irreducible
  component or (if $\shift_{\#i}$ was used, and the result was
  reducible) the intersection of the two components. Then:
  \begin{enumerate}
  \item Each $\shift_{\#i} X_{i-1}$ is reduced (not just generically reduced).
  \item It is again true that each $\shift_{\#i} X_{i-1}$ has at most
    two components. If there are two, then their intersection is
    reduced and irreducible.
    Again, $X_{n\choose 2}$ is an opposite Schubert variety.  
  \item Each $X_i$ is a ``positroid variety'', which implies \cite{KLS} that
    it is normal, Cohen-Macaulay, has rational singularities, 
    is defined by the vanishing of Pl\"ucker coordinates,
    and has other admirable qualities described in \cite{KLS}.
  \end{enumerate}
\end{Theorem}

Despite the fact that Vakil did not study the sweep operations 
(relevant for equivariant cohomology) 
or the intersections (relevant for $K$-theory),
we will call any variety $\{X_i\}$ constructed in the theorem 
above a \defn{Vakil variety}. So we have, in increasing order of generality,
$\{$Schubert varieties$\} 
\subseteq 
\{$Richardson varieties$\} 
\subseteq 
\{$Vakil varieties$\} 
\subseteq \{$positroid varieties$\}$. In \S \ref{ssec:interval} we
will define yet another class in between the last two.

Our cohomological application of this theorem is to extend Vakil's
rule for the cohomology product to one in equivariant $K$-theory.

\subsection{Puzzles}\label{ssec:puzzleintro}

A \defn{puzzle triangle} is just an equilateral triangle of
side-length $n$, oriented like $\Delta$ (not $\nabla$). 
It has $n+2 \choose 2$ \defn{puzzle vertices}, connected by
$3{n+1 \choose 2}$ \defn{puzzle edges} parallel to the sides, whose 
directions we will call by approximate compass directions E/W, NE/SW, NW\!/SE.
In particular, we may refer to the $n$ \defn{rows} of a puzzle
counted from the top down,
and its \defn{NW/SE columns} and \defn{NE/SW columns}, 
each counted from left to right.
Consider \defn{unlabeled puzzle paths} $\gamma$ that 
(as in figure \ref{fig:initAndTerm}) traverse puzzle edges
\begin{itemize}
\item starting at the top vertex of the puzzle, then
\item head Southeast some distance along the Northeast side of the puzzle,
\item head Southwest some distance through the puzzle,
\item jog one optional step Southeast along an edge called the \defn{kink},
\item continue Southwest until they hit the bottom edge,
\item and go West until they hit the Southwest corner.
\end{itemize}
Since we think of $\gamma$ as a directed path, we will talk about one
edge of $\gamma$ being ``after'' another edge, as traversed
in the order above.
There are two puzzle paths that stay entirely on the boundary:
the \defn{initial path} which follows the NE edge then bottom edge, and
the \defn{final path} which follows the NW edge. 
While we described the kink as optional, except for final paths $\gamma$
we can always take the last SE step to be the kink. 

\begin{figure}[htbp]
  \centering
  \epsfig{file=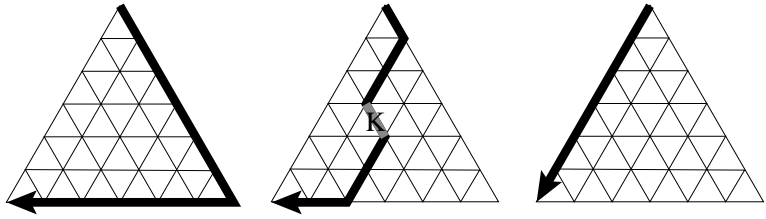,width=4in}
  \caption{The initial path, a more general puzzle path with a $K$ 
    indicating the kink, and the final path.}
  \label{fig:initAndTerm}
\end{figure}

Label the edges along $\gamma$, each with one of four possible labels:
\begin{itemize}
\item $0$
\item $1$
\item $R$ (for Rhombus) -- this may not occur on the outer boundary of
  the puzzle
\item $K$ (for $K$-theory) -- this may only occur on the kink,
  and not on the outer boundary of the puzzle.
\end{itemize}
Only certain labelings are allowed; we detail the conditions in 
\S \ref{sec:variety}. A \defn{puzzle path} will refer to one with
an allowed labeling. Some may be seen in figure \ref{fig:puzzlepathexs}
on p\pageref{fig:puzzlepathexs}. To refer to labeled edges on $\gamma$,
we will talk about the $\dash 0$ edges, the $\bslash R$ edges, etc.

To each puzzle triangle with a (labeled) puzzle path $\gamma$, we will
explain in \S \ref{sec:variety} how to select certain horizontal edges in 
the puzzle triangle, with which to define an upper triangular partial
permutation matrix and, eventually, a Vakil subvariety of the Grassmannian.

Each step of Vakil's geometric algorithm will then correspond to a
small change in $\gamma$, with the whole process going from the
initial path to the final path.  We will record this process by
placing ``puzzle pieces'' in a separate copy of the triangle. 
The proof that Vakil's degenerative geometry is captured by the
combinatorics of the puzzle pieces will be theorem \ref{thm:filling}.

Puzzle pieces come in three types:
\begin{itemize}
\item triangles, which may be rotated: \epsfig{file=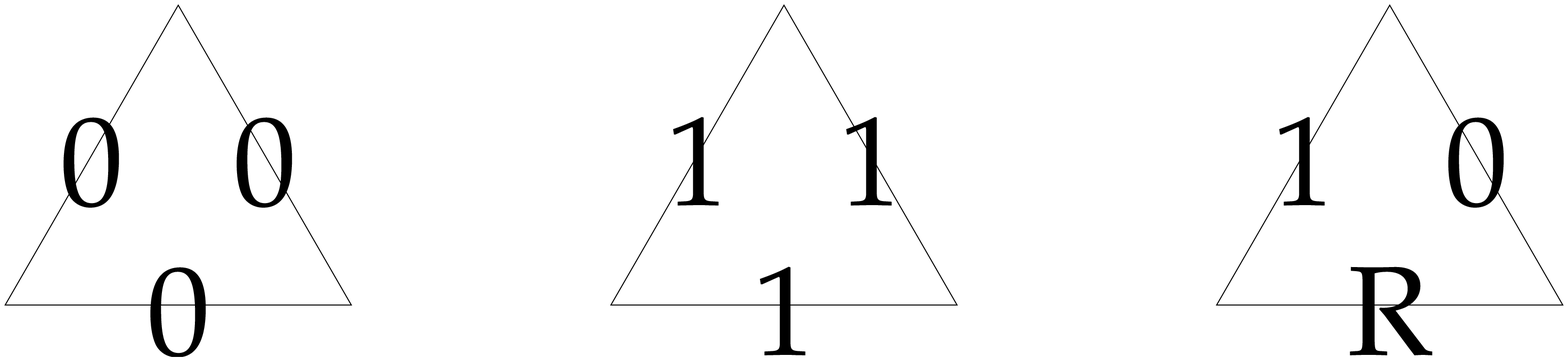,height=.5in}
\item the \defn{equivariant rhombus} \epsfig{file=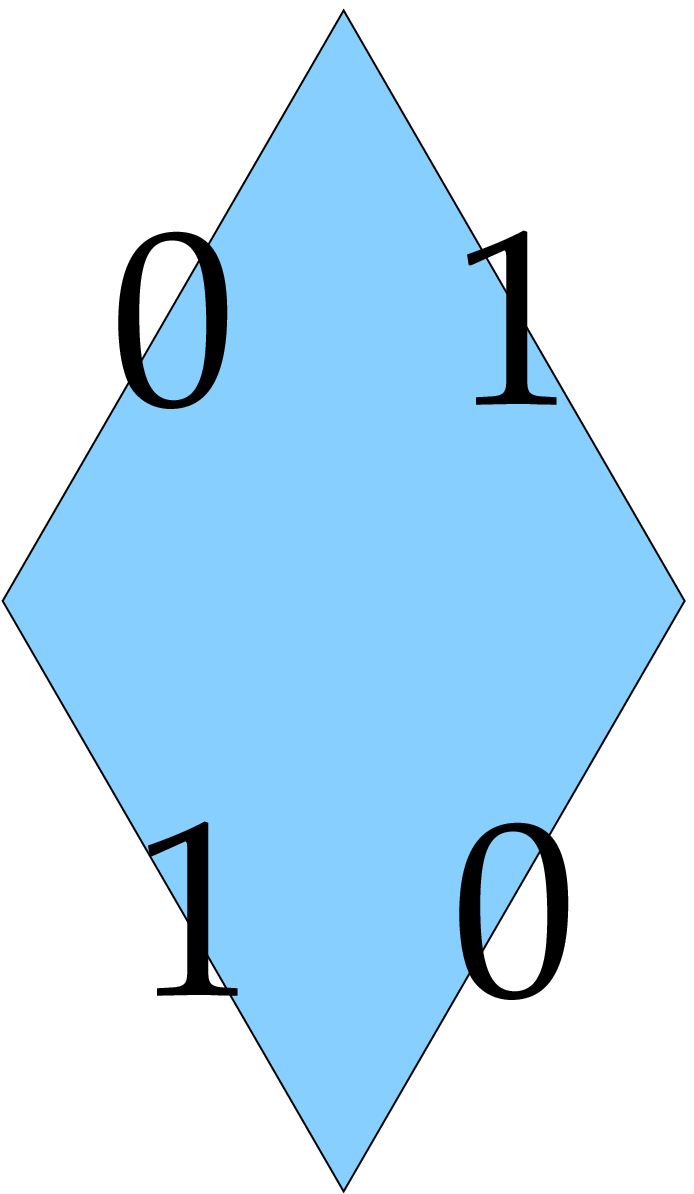,height=.8in},
  and
\item the \defn{top, middle, and bottom $K$-rhombi} 
  \epsfig{file=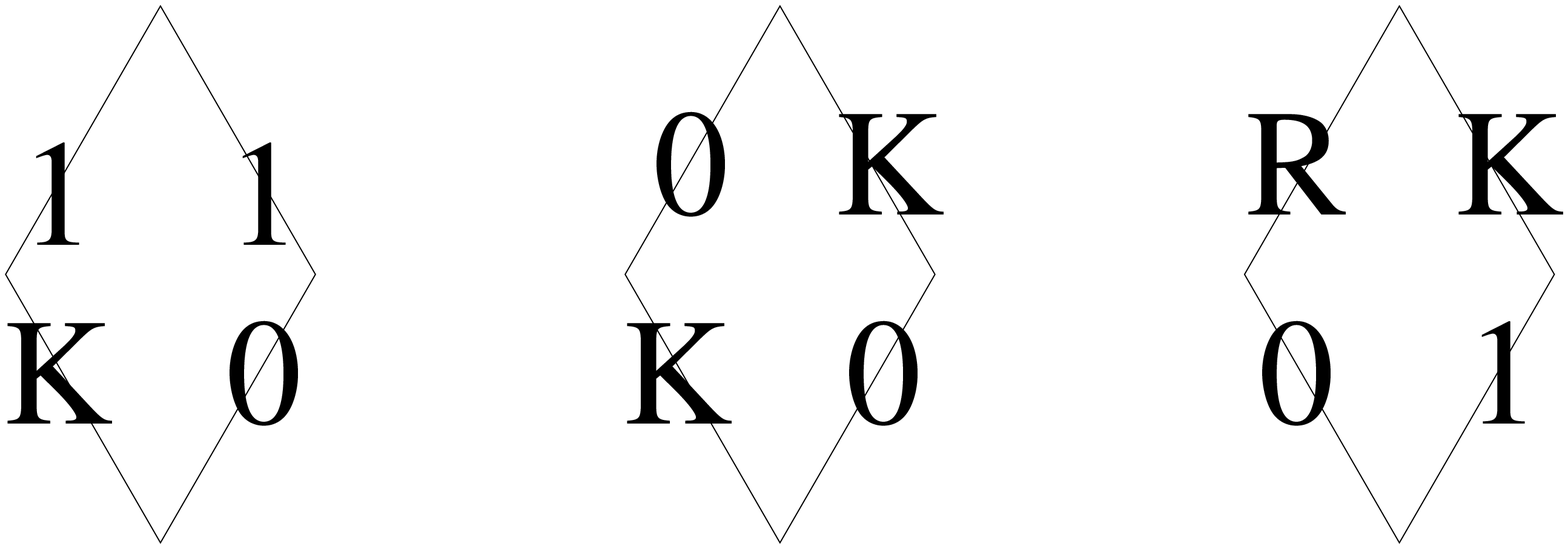,height=.8in}.
\end{itemize}
We will often want a \defn{puzzle rhombus} to refer not only to the
equivariant and $K$-rhombi but also to a $\Delta$ piece atop a $\nabla$ piece,
with the labels on the horizontal edges matching.
In this paper, ``puzzle rhombi'' will always have this vertical orientation.

Define, almost, a \defn{puzzle} to be a tiling by these pieces of a
large triangle such that edge labels match up, and with only $0$,$1$
labels on external edges. 
\junk{
In particular, across the bottom edge are
$n$ triangles, and the rest of the pieces are $n\choose 2$ rhombi.  }
We say ``almost'' because there are two non-local conditions concerning the
placement of $K$-edges, each of which appears on the kink of a (unique)
puzzle path $\gamma$: (1) if $\bslash K$ is due NE of a $\dash 1$, there
must be a $\fslash R$ along $\gamma$ somewhere between them, 
and (2) if indeed $\bslash K$ is NE of an $\fslash R$,
there must be a $\fslash 1$ along $\gamma$ somewhere between them.

If we disallow $K$-rhombi, then we can glue the triangles with
$R$-edges together in pairs to make the rhombi in the \cite{KTW,KT}
formulations of puzzles. 
If we instead glue the $K$-rhombi together along their $\bslash K$s, 
in each aggregate the ``top'' $K$-rhombus will be on top, 
the ``middle'' $K$-rhombus occurring several times in the middle 
(possibly zero), and the ``bottom'' $K$-rhombus on bottom.

In figure \ref{fig:finalpuzex} on p\pageref{fig:finalpuzex} we give
all the puzzles with $0101$ and $1010$ on the NE and S sides.

For each horizontal edge $e$ in the puzzle,
let $i(e)$ denote its NE/SW column and $j(e)$ its NW/SE column.
So if we drop lines Southwest and Southeast from the edge,
they point to the $i(e)$th and $j(e)$th edges on the bottom;
we may refer to the edge or the vertical rhombus it bisects as
being \defn{in position $(i(e),j(e))$}.
If we consider the horizontal edges one NE/SW column at a time, 
rightmost column to leftmost, then down each column
(but skipping the bottom edges),
their $i(e),j(e)$ correspond to the shifts $(i(e)\to j(e))$ 
in Vakil's degeneration order. See figure \ref{fig:degenorder}.

\begin{figure}[htbp]
  \centering
  \epsfig{file=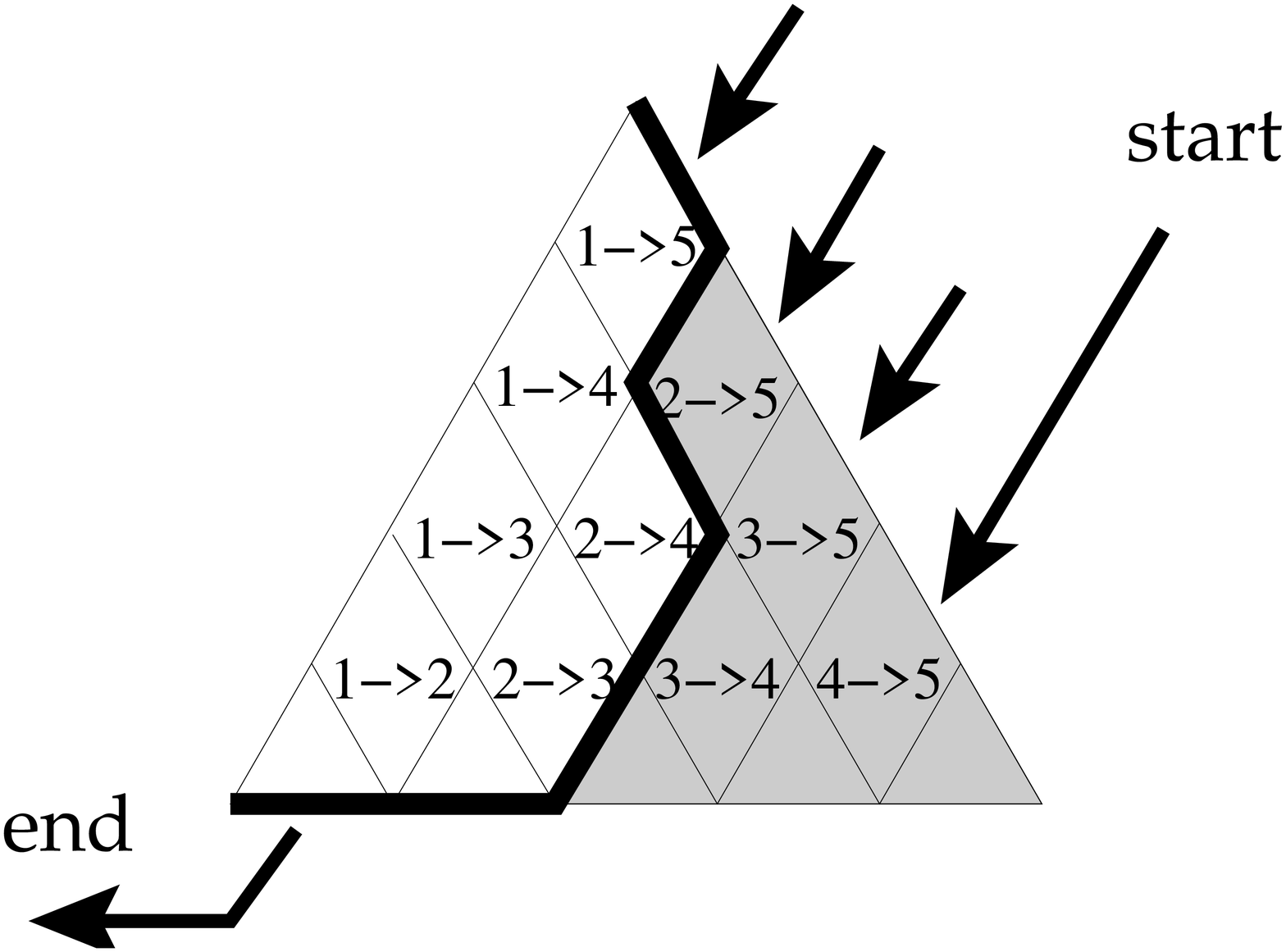,height=2in}
  \caption{Vakil's degeneration order of shifts, 
    thought of as a filling order on the rhombi in the puzzle.
    The boundary between the rhombi filled so far, and those yet to filled,
    is an unlabeled puzzle path. In the picture above, $2\to 4$ is the
    next to be filled.}
  \label{fig:degenorder}
\end{figure}

In the following theorem, we consider pairs $\gamma$, $\gamma'$ 
of puzzle paths whose symmetric difference $p$
is either a $\Delta$ piece or two triangles stacked in a vertical rhombus,
as in figure \ref{fig:added}.
In this situation, say that \defn{$p$ added to $\gamma$} (on the right of $p$) 
\defn{gives $\gamma'$} (on the left of $p$), where $p$ is the one or two 
puzzle pieces.
It is easy to see from the allowed shapes of $\gamma,\gamma'$ that there
is a unique location one might add some $p$ to $\gamma$, either filling in the 
triangle at the bottom of a NE/SW column or moving the kink SW one rhombus.

\begin{figure}[htbp]
  \centering
  \epsfig{file=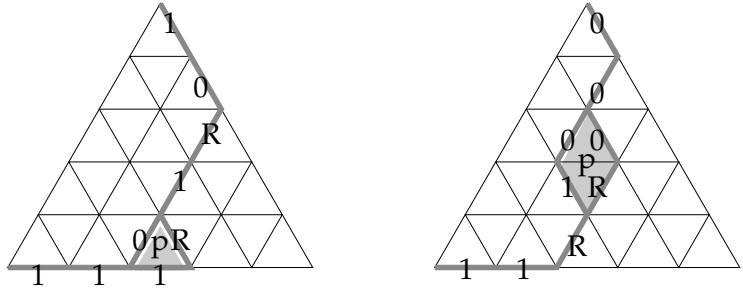,height=1.5in}
  \caption{Each picture contains the superposition of two puzzle paths
    $\gamma$ and $\gamma'$ agreeing away from a puzzle triangle or rhombus, $p$,
    which added to $\gamma$ (on the right of $p$) 
    gives $\gamma'$ (on the left of $p$).}
  \label{fig:added}
\end{figure}

\newcommand\barPi{{\overline \Pi}}

\begin{Theorem}\label{thm:puzschematic}
  To each puzzle path $\gamma$, there is a way given in \S \ref{sec:variety} 
  to associate an ``interval rank variety'' $\barPi_r \subseteq M_{k\times n}$,
  defined by rank conditions on intervals of columns,
  whose associated ``interval positroid variety''
  $\Pi_r := GL(k) \dom (\barPi_r \cap \St_{k,n})$
  in the Grassmannian will turn out to be a Vakil variety.

  If $\gamma$ is initial, $\Pi_r$ is a Richardson variety. 
  If $\gamma$ is terminal, $\Pi_r$ is an opposite Schubert variety. 

  If $\gamma$ is not terminal, take its last SE edge to be the kink.
  If the next step $\sigma$ is due West, there exists a unique triangular 
  puzzle piece to add to $\gamma$, obtaining a new puzzle path $\gamma'$.
  This $\gamma'$ has the same associated interval rank variety.

  If the next step $\sigma$ is SW, and the kink and $\sigma$ are {\em not}
  labeled $0$ and $1$ respectively, there exists a unique puzzle
  rhombus to add to $\gamma$, obtaining a new puzzle path $\gamma'$.
  This $\gamma'$ has the same associated interval rank variety.
  This situation occurs iff $\shift_{i(e)\to j(e)} \Pi_r = \Pi_r$,
  where $e$ is the horizontal edge crossing the rhombus.

  If the next step $\sigma$ is SW, and the kink and $\sigma$ {\em are}
  labeled $0$ and $1$ respectively, there exist multiple puzzle 
  rhombi to add to $\gamma$, obtaining new puzzle paths $\gamma'$.
  Each such $\gamma'$ has a different associated interval rank variety
  (and all are different from that of $\gamma$).
  This situation occurs iff $\shift_{i(e)\to j(e)} \Pi_r \neq \Pi_r$,
  where $e$ is the horizontal edge crossing the rhombus.
  Indeed $\sweep_{i(e)\to j(e)} \Pi_r$ is the Vakil variety constructed
  from adding the equivariant piece to $\gamma$,
  whereas $\shift_{i(e)\to j(e)} \Pi_r$ is the union of the 
  Vakil varieties associated to the other possible additions.
\end{Theorem}

Readers wishing to see a detailed example may jump directly to \S \ref{sec:ex}.

\subsection{Positivity and puzzle statements in various
  cohomology theories}\label{ssec:positivity}

Let $\calP_{\lambda\mu}^\nu$ be the set of puzzles with labels
$\lambda$ on the NW side, $\mu$ on the NE side, $\nu$ on the S side,
all read left to right. 
For each of the cohomology theories $E^*$ discussed below, and each 
rhombus puzzle piece,
we associate an element $\Phi(E^*,\rho)$ of $E^*(pt)$ 
(possibly $0$), so that the formula 
$\sum_{P\in \calP_{\lambda\mu}^\nu} \prod_{\rho\in P} \Phi(E^*,\rho)$
will turn out to compute a coefficient of interest.

\subsubsection{Ordinary cohomology}
As already mentioned, the Schubert cycles $\{X_\lambda\}$ define a
$\integers$-basis of the cohomology ring of the Grassmannian. 
In this theory, the \defn{Littlewood-Richardson coefficients} 
$\{c_{\lambda\mu}^\nu \in \naturals\}$
show up in two expansions:
$$ [X_\lambda] [X_\mu] = \sum_\nu c_{\lambda\mu}^\nu [X_\nu], \qquad
[X_\mu^\nu] = \sum_{\lambda} c_{\lambda\mu}^\nu [X^\lambda]. $$
In essence, it is the latter expansion that Vakil studies, largely because 
the intersection $X_\mu \cap X^\nu$ is transverse and the intersection
$X_\lambda \cap X_\mu$ is not.

Both expansions are consequences of the alternate definition 
$$ c_{\lambda\mu}^\nu = \int_{\Grkn} [X_\lambda] [X_\mu] [X^\nu] $$
and the dual-basis relation 
$\int_{\Grkn} [X_\lambda] [X^\nu] = \delta_{\lambda\nu}$.

\begin{Theorem}\cite{KTW}\label{thm:hpuz}
  Let the factors $\Phi(H^*,\rho)$ be $0$ for equivariant and $K$-rhombi.
  Then $c_{\lambda\mu}^\nu 
  = \sum_{P\in \calP_{\lambda\mu}^\nu} \prod_{\rho\in P} \Phi(H^*,\rho)
  =$ the number of puzzles $P \in \calP_{\lambda\mu}^\nu$ using only triangles.
\end{Theorem}

\subsubsection{Equivariant cohomology}\label{sssec:HT}

Since the Schubert and Richardson varieties are invariant under
the action of the torus $T \leq GL(n)$ of diagonal matrices,
they define also a basis of the equivariant cohomology ring $H^*_T(\Grkn)$,
considered as a module over 
$H^*_T(pt) \iso \integers[\yy] := \integers[y_1,\ldots,y_n]$.
We do not need to introduce new notation; the ``equivariant numbers''
$c_{\lambda\mu}^\nu(\yy) \in \integers[\yy]$ specialize to the
ordinary numbers $c_{\lambda\mu}^\nu \in \integers$ by specializing
each $y_i \mapsto 0$. In particular, $c_{\lambda\mu}^\nu(\yy) \neq 0$
implies $|\lambda|+|\mu| \geq |\nu|$,
and if $|\lambda|+|\mu| = |\nu|$ then 
$c_{\lambda\mu}^\nu(\yy) = c_{\lambda\mu}^\nu$.

The Schubert and opposite Schubert varieties are related 
$$ X^\lambda = w_0 \cdot X_{\lambda\text{ reversed}} $$
by the
\defn{long element} $w_0 = (1\leftrightarrow n)(2 \leftrightarrow n-1)\cdots$
of $S_n$, and hence define the same element of the cohomology ring.
For that reason, one may wonder why we used both $[X_\lambda]$s and
$[X^\nu]$s in the equations in \S \ref{sssec:HT}, 
rather than stating everything in one basis.
This is because the Schubert and opposite Schubert varieties 
do {\em not} define the same elements in equivariant cohomology,
and only when written in the form above do the relations extend
to equivariant cohomology.

It was proven abstractly for generalized flag manifolds $G/P$ \cite{Graham}, 
and combinatorially for $\Grkn$ \cite{KT}, 
that the equivariant numbers $\{c_{\lambda\mu}^\nu(\yy)\}$ can be
written as $\naturals$-combinations of products of distinct positive roots
$y_i - y_j$, $i>j$. 
(The reference \cite{Graham} makes the weaker claim that
$c_{\lambda\mu}^\nu(\yy)$ is an $\naturals$-combination of products 
of simple roots, but the proof there gives this more precise result.)

\begin{Theorem}\cite{KT}\label{thm:htpuz}
  Let the factors $\Phi(H^*_T,\rho)$ be $0$ for the $K$-pieces.
  For an equivariant rhombus, the factor is 
  $y_{j(e)}-y_{i(e)}$.
  Then $c_{\lambda\mu}^\nu(\yy)
  = \sum_{P\in \calP_{\lambda\mu}^\nu} \prod_{\rho\in P} \Phi(H^*_T,\rho)$.

  Equivalently, one can work with only triangles and their rotations,
  plus the equivariant piece (no rotations),
  which is very nearly the viewpoint of \cite{KT}.
\end{Theorem}

\subsubsection{(Nonequivariant) $K$-theory}

Let $[\calO_\lambda], [\calO^\nu]$ denote the classes in $K(\Grkn)$ of
the structure sheaves $\calO_\lambda, \calO^\nu$ of the Schubert and
opposite Schubert varieties.  These are not dual bases:
$$ \Kint [\calO_\lambda] [\calO^\nu] = 
\begin{cases}
  1 & \text{if } X_\lambda \cap X^\nu \neq \emptyset, \qquad
  \text{i.e. $\lambda\leq\nu$ in \defn{Bruhat order}} \\
  0 & \text{otherwise.}
\end{cases}
$$
Here $\Kint : K(\Grkn) \to K(pt) \iso \integers$ denotes the
pushforward to a point in $K$-theory, giving the ``holomorphic Euler
characteristic'' of a sheaf.  

Consequently, there is another basis of $K(\Grkn)$ to consider;
the dual basis $\{ [\xi^\nu] \}$ satisfying
$\Kint [\calO_\lambda] [\xi^\nu] = \delta_{\lambda\nu}$.
(Right now ``$[\xi^\nu]$'' is just a $K$-class; in a moment we will
define an actual sheaf $\xi^\nu$.)
Using the known M\"obius function of the Bruhat order, one can show that
$$ [\calO_\lambda] = \sum_{\nu\geq\lambda} [\xi_\lambda], \qquad
[\xi_\lambda] = \sum_{\nu\geq\lambda} (-1)^{|\nu|-|\lambda|} [\calO_\lambda]. $$
It is a pleasant fact \cite[Proposition 2.1]{GrahamKumar} that this
$K$-class $[\xi^\nu]$ is actually the $K$-class of a sheaf $\xi^\nu$,
the subsheaf of $\calO^\nu$ consisting of functions vanishing on
$\partial X^\nu := \bigcup_{\lambda < \nu} X^\lambda$.

If we define the coefficients $g,e$ by
$$ [\calO_\lambda] [\calO_\mu] = \sum_\nu g_{\lambda\mu}^\nu [\calO_\nu],\qquad
[\calO_\mu] [\calO^\nu] = \sum_\lambda e_{\lambda\mu}^\nu [\calO^\lambda] $$
then
$$ g_{\lambda\mu}^\nu = \Kint [\calO_\lambda] [\calO_\mu] [\xi^\nu],\qquad
e_{\lambda\mu}^\nu = \Kint [\calO_\mu] [\calO^\nu] [\xi_\lambda]
\qquad
\text{so } 
e_{\lambda\mu}^\nu = g_{\nu\text{ reversed}, \mu}^{\lambda\text{ reversed}}.
$$
We will extend Vakil's techniques to study the $e$ coefficients, and
thereby obtain the $g$ coefficients as well.

The coefficients $g_{\lambda\mu}^\nu \in \integers$ turn out to be
nonnegative once multiplied by $(-1)^{|\nu|-|\lambda|-|\mu|}$, as was
first shown combinatorially in the Grassmannian case in \cite{Buch}, 
and then geometrically for arbitrary $G/P$ in \cite{BrionPos}.
The condition $g_{\lambda\mu}^\nu \neq 0$ 
implies that $|\lambda|+|\mu| \leq |\nu|$ -- the opposite inequality
we had for equivariant cohomology --
and if $|\lambda|+|\mu| = |\nu|$ then
$g_{\lambda\mu}^\nu = e_{\lambda\mu}^\nu = c_{\lambda\mu}^\nu$.

Now that we have another basis $\{[\xi_\lambda]\}$ of $K(\Grkn)$, 
we can consider also its structure constants
$$ [\xi_\lambda] [\xi_\mu] = \sum_\nu p_{\lambda\mu}^\nu [\xi_\nu] $$
(though it is perhaps a bit weird to do so, as $1$ is not an element
of this basis).  Once again, these are nonnegative once multiplied by
$(-1)^{|\nu|-|\lambda|-|\mu|}$ \cite[Remark 3.7]{GrahamKumar}.

So far everything in this discussion of $K$-theory holds for 
Schubert classes on arbitrary flag manifolds $G/P$. 
We now make use of a special property characterizing {\em minuscule} $G/P$:
the two bases have a further relation
$[\xi_\lambda] = [\calO_\lambda] (1 - \square)$, where $\square$ denotes the
$K$-class of the (unique) Schubert divisor. (On Grassmannians, this fact can
be found in \cite[\S 8]{Buch}, where it is used to show a $3$-fold symmetry
of the $g$ coefficients.) Then
$$ \Kint [\calO_\lambda] [\calO_\mu] [\xi^\nu]
 = \Kint [\calO_\lambda] [\calO_\mu] [\calO^\nu] (1 - \square)
 = \Kint [\calO_\mu] [\calO^\nu] [\xi_\lambda] 
$$
so $g_{\lambda\mu}^\nu = e_{\lambda\mu}^\nu$. We also obtain the relation
\begin{eqnarray*}
  p_{\lambda\mu}^\nu 
  = \Kint [\xi_\lambda] [\xi_\mu] [\calO^\nu]
  = \Kint [\calO_\lambda] (1 - \square) [\calO_\mu] (1 - \square) [\calO^\nu]
  = \Kint [\calO_\lambda] (1 - \square) [\calO_\mu] [\xi^\nu]
  = g_{\lambda\mu}^\nu - g_{\square\lambda\mu}^\nu
\end{eqnarray*}
and both of the latter terms (the second one, a structure constant for
a triple product) have the right sign for $p_{\lambda\mu}^\nu$. 
Hence, in the case $G/P$ minuscule, the positivity property of the $p$
coefficients follows from that of the $g$.

\begin{Theorem}(An analogue of \cite[Theorem 3.6]{Vakil}.)\label{thm:kpuz}
  Let the factors $\Phi(K,\rho)$ be $0$ for the equivariant piece,
  $-1$ for the top $K$-piece, and $1$ for the others.

  Then $g_{\lambda\mu}^\nu(\yy)
  = \sum_{P\in \calP_{\lambda\mu}^\nu} \prod_{\rho\in P} \Phi(K,\rho)
  = (-1)^{|\nu|-|\lambda|-|\mu|}\ \#\{$puzzles $P$ using only these pieces$\}$.
\end{Theorem}

\subsubsection{Equivariant $K$-theory}\label{sssec:KT}

Our reference for this subject is \cite{AGM}.

The base ring $K_T(pt)$ for $T$-equivariant $K$-theory is the
representation ring of $T$, and isomorphic to a Laurent polynomial ring. 
Since we use $y_1,\ldots,y_n$ to denote an {\em additive} basis of the
weight lattice of $T$, we will instead use $\exp(y_1),\ldots,\exp(y_n)$ 
to denote the corresponding elements of 
$K_T(pt) \iso \integers[e^\yy]:= \integers[\exp(\pm y_1),\ldots,\exp(\pm y_n)]$.
The sheaves $\calO_\lambda, \xi_\lambda, \calO^\nu, \xi^\nu$ are 
$T$-equivariant, and so define classes in $K_T(\Grkn)$,
for which we use the same notation
$[\calO_\lambda], [\xi_\lambda], [\calO^\nu], [\xi^\nu]$ as before.

The structure constants 
$e_{\lambda\mu}^\nu,p_{\lambda\mu}^\nu,g_{\lambda\mu}^\nu$ generalize to 
$e_{\lambda\mu}^\nu(e^\yy),p_{\lambda\mu}^\nu(e^\yy),g_{\lambda\mu}^\nu(e^\yy)$,
and again the latter specialize to the former under $\yy\mapsto 0$,
$e^\yy \mapsto 1$. Each family has the same positivity statement:
$$ e_{\lambda\mu}^\nu(e^\yy),p_{\lambda\mu}^\nu(e^\yy),g_{\lambda\mu}^\nu(e^\yy)
\quad \in \quad
(-1)^{|\lambda|+|\mu|-|\nu|} \
\naturals\big[ \{ \exp(y_i - y_j) - 1 \} \big]_{i>j} $$
but these appear to require three different proofs
\cite[Corollaries 5.1-5.3]{AGM}. I do not know (even conjecturally)
the proper analogue of the ``products of {\em distinct} positive roots''
property mentioned at the end of \S \ref{sssec:HT}.

Vakil's geometric techniques generalize most easily to studying
the $e_{\lambda\mu}^\nu$ structure constants, and we confine ourselves
to that problem in this paper.

\junk{
We now give the resulting rule for the structure constants 
$e_{\lambda\mu}^\nu(e^\yy) \in K_T(pt)$ from 
$$ [\calO_\mu] [\calO^\nu] = \sum_\lambda e_{\lambda\mu}^\nu [\calO^\lambda], $$
extending rules from \cite{KT,Vakil,Buch}
in $T$-equivariant cohomology and ordinary $K$-theory.
Our rule is manifestly positive in the sense of \cite{AGM}, detailed
in \S \ref{sssec:KT}.
}
\begin{Theorem}\label{thm:ktpuz}
  Let $\beta(\rho) = y_{i(e)} - y_{j(e)}$ for $\rho$ a vertical rhombus,
  and $e$ the edge bisecting it.
  Let the factors $\Phi(K_T,\rho)$ be as follows:
  $$ \Phi(K_T,\rho) = 
  \begin{cases}
    1 - e^{\beta(\rho)} & \text{if $\rho$ is an equivariant piece} \\
    e^{\beta(\rho)}
    &\text{if $\rho$ is another rhombus with $1,0$ on its right side}\\
    -1 & \text{if $\rho$ is the top $K$-piece}\\
    1 & \text{otherwise.}
  \end{cases}
  $$

  Then $e_{\lambda\mu}^\nu(e^\yy)
  = \sum_{P\in \calP_{\lambda\mu}^\nu} \prod_{\rho\in P} \Phi(K_T,\rho).$
\end{Theorem}

To be sure that these formul\ae\ have the desired positivity properties,
we give a lemma, which can be proved (though we won't do so)
by the techniques from \cite[\S 4]{KT}.

\begin{Lemma}\label{lem:inversions}
  Let $\Delta$ be a puzzle, with $\lambda$,$\mu$,$\nu$ the strings of
  labels on the NW, NE, S sides respectively, all read left-to-right.
  Then 
  $$ |\nu| + \#\{ \text{equivariant rhombi in }P\} 
  = |\lambda|+|\mu| + \#\{\text{top $K$-rhombi in }P\}. $$
\end{Lemma}

\subsection{Interval rank varieties}\label{ssec:interval}

Given a matrix $M\in M_{k,n}$, associate an upper triangular 
\defn{interval rank matrix} 
$r(M)$ by 
$$ r(M)_{ij} :=
\begin{cases}
  rank(M_{[i,j]}) & \text{if $i\leq j$} \\
  0 & \text{if $i>j$} 
\end{cases}
$$

\begin{Theorem}\label{thm:rankmatrices}
  \begin{itemize}
  \item For any $M \in M_{k,n}$, there exists a unique upper triangular 
    $n\times n$ partial permutation%
    \footnote{meaning, at most one $1$ in each row and column} 
    matrix $J(r)$ such that
    $$ r(M)_{ij} 
        = |[i,j]| - \#\{\text{$1$s in $J(r)$ southwest of $(i,j)$}\}. $$
  \item Every upper triangular $n\times n$ 
    partial permutation matrix with at least $n-k$ $1$s arises this way.
  \item If we fix an interval rank matrix $r$ that actually arises for some $M$,
    the \defn{interval rank variety}
    $$ \barPi_r := \{ N : r(N) \leq r \text{ entrywise} \} $$
    is isomorphic to a certain ``Kazhdan-Lusztig variety'' in 
    a flag manifold. Hence it is reduced, irreducible, normal,
    Cohen-Macaulay, and has rational singularities, and there is a good
    formula for its $T$-equivariant Hilbert series.
  \end{itemize}
\end{Theorem}

The quotient
$$ \Pi_r := GL(k) \dom (\barPi_r \cap \St_{k,n}) $$
is a special case of a ``positroid'' subvariety of the Grassmannian.
Positroid varieties are defined by rank conditions
on all {\em cyclic} intervals of columns, 
i.e. including $i,i+1,\ldots,n-1,n,1,2,\ldots,j$.
We studied these in \cite{KLS}, where we showed they are reduced, irreducible,
normal, and Cohen-Macaulay with rational singularities.
Unfortunately, we don't know this upstairs in $M_{k,n}$
(just in $\St_{k,n}$), when cyclic conditions are used,
so we make use of the connection to Kazhdan-Lusztig 
varieties in the flag manifold to give an independent proof.

We will call these $\{\Pi_r\}$ \defn{interval positroid varieties}.
It will turn out that each Vakil variety is of the form
$\Pi_r$ for some $r$.

In \cite{HostenSullivant} they determine the components of the
subscheme of $M_{k\times n}$ defined by asking that each connected
$k\times k$ minor vanish.  Via the connection to positroid varieties,
one can show that each of these components is an interval rank variety.

\subsection*{Acknowledgments}
This paper would not have been possible without the early participation
of Ravi Vakil. It was in an attempt to understand his work that I began
to look into positroid varieties, and have learned so much about them 
from discussions with Thomas Lam, David Speyer, and Michelle Snider.  

\section{Interval rank varieties}\label{sec:intvpositroid}

Let $B_-$, respectively $B_+$, denote the groups of lower, respectively
upper, triangular matrices in $GL(N)$. 
A \defn{Schubert variety} in the flag manifold $B_- \dom GL_N$ is
the closure
$$ X_\pi := \overline{B_- \pi B_+} / B_+ $$
where $\pi$ is a permutation matrix. 
\junk{They are related to the 
  Grassmannian Schubert varieties in the introduction by identifying
  the Grassmannian with $P_- \dom GL(N)$, where $P_-$ is the 
  block-lower-triangular matrices with block sizes }
These are well-known to be Cohen-Macaulay with rational singularities
(see e.g. \cite{BrionFlags}). An \defn{opposite Schubert cell} is an
orbit of $B_+$,
$$ X^\rho_\circ :=  B_+ \rho B_+ / B_+, $$
and a \defn{Kazhdan-Lusztig variety} (terminology from \cite{WY})
is the intersection
$$ X^\rho_{\pi\circ} := X_\pi \ \cap\ X^\rho_\circ. $$
It is of dimension $\ell(\rho)-\ell(\pi)$, where
$$ \ell(\pi) := \#\{(i,j) : i<j, \pi(i)>\pi(j) \}. $$
This variety $X^\rho_{\pi\circ}$
is used to study the singularities of $X_\pi$ near the point 
$\rho B_+ / B_+$. These affine varieties have the same good properties
as the Schubert varieties, and have nice degenerations to 
unions of coordinate spaces \cite[\S 7.3]{frobAffine}.

It will be convenient to study $X_\pi$ via its
\defn{matrix Schubert variety} \cite{Fulton92,GrobnerGeom}:
$$ \barX_\pi := \overline{B_- \pi B_+} \subseteq M_{N\times N} $$
Fulton \cite{Fulton92} determined the equations defining $\barX_\pi$:
$$ \barX_\pi := \left\{ M \in M_{N\times N} 
        \ :\ rank(M_{\leq i,\leq j}) \leq rank(\pi_{\leq i,\leq j}),
      i,j\leq n \right\} $$
where $N_{\leq i,\leq j}$ denotes the upper left $i\times j$ submatrix.
It is enough to take $(i,j)$ in Fulton's \defn{essential set}, the
Southeast corners of $\pi$'s Rothe diagram.
Fulton also proves (after \cite[lemma 6.1]{Fulton92})
that $\barX_\pi$ is itself a Kazhdan-Lusztig variety
for $GL(2N)$, without using that language. 

\begin{proof}[Proof of theorem \ref{thm:rankmatrices}]
  An interval rank matrix is easily seen to satisfy the following properties:
  \begin{enumerate}
  \item The diagonal entries are $0$ or $1$.
  \item Each entry is either $0$ or $1$ more than the entries 
    West and South of it.
  \item If $r(M)_{ij} = r(M)_{i-1,j} = r(M)_{i,j+1}$, then 
    $r(M)_{ij} = r(M)_{i-1,j+1}$.
  \end{enumerate}
  Let $J = J(r)$ be the upper triangular matrix with $1$ at $(i,j)$
  iff $r(M)_{ij} = r(M)_{i,j-1} = r(M)_{i+1,j} = r(M)_{i+1,j-1}+1$,
  and $0$ otherwise. Then $J$ is a partial permutation matrix, 
  and $r(M)_{ij} = |[i,j]| - \#\{\text{$1$s in $\pi$ southwest of $(i,j)$}\}.$
  We refer to \cite[corollaries 3.10-3.12]{KLS} for the proof of a 
  similar but but more general statement.

  We will show that $\barPi_r$ is isomorphic to a Kazhdan-Lusztig variety
  $X^\rho_{\pi\circ}$ in $GL(k+n)/B$. 
  The upper index $\rho$ will not depend on $r$: in one-line notation 
  it is $n\!+\!1\ n\!+\!2\ \ldots n\!+\!k\ 1\ 2\ \ldots\ n$, of length $kn$.
  There is a handy subset $C \subseteq GL(k+n)$ that projects
  isomorphically to $X^\rho_\circ = B_+ \rho B_+ / B_+$:
  $$ C := \left\{ \left[
      \begin{array}{cc}
        N & Id_k \\
        Id_n & 0 
      \end{array}
    \right] : N \in M_{k\times n} \right\}  \qquad
  \text{where the $Id_k,Id_n$ are identity matrices.} 
  $$
  We still need to define $\pi$ from $J$, which we recall has
  at least $n-k$ $1$s, thus at most $k$ empty rows and $k$ empty columns.
  Loosely speaking,
  put $J$ in the lower left of $\pi$ and extend it to a permutation 
  matrix in the unique way with fewest inversions:
  $$ \pi = \left[
      \begin{array}{cc}
        A_1 & {0\atop Id_{s}} {0\atop 0} \\
        J   & A_2 
      \end{array}    \right], \qquad s = rank(J)-(n-k).
  $$
  In more detail, $A_1$ is the $k\times n$
  partial permutation matrix whose $j$th row ($j\leq k-s$) 
  has a $1$ in the $j$th empty column of $J$, and $A_2$ is the $n\times k$ 
  partial permutation matrix whose $j$th column has a $1$ in the $j$th
  empty row of $J$. It is easy to see that $\pi$'s Rothe diagram lies
  in the first $n-k$ columns, has no essential boxes above
  the $k$th row, and has no boxes in the lower triangle of the $J$ square. 
  In particular, Fulton's description of $\barX_\pi$ implies
  \begin{eqnarray*}
     M \in \barX_\pi \Longleftrightarrow
     \forall 0 \leq i < j\leq n,
     rank(M_{\leq k+i,\leq j}) 
     &\leq&  \#\{j'\leq j : \text{column $j'$ of $J$ is zero} \}
     + rank(J_{\leq i,\leq j}) \\
     &=& j - rank(J_{> i,\leq j}) \\
     &=& i + \big| [i+1,j] \big| - rank(J_{\geq i+1,\leq j}) \\
     &=& i + r_{i+1,j}.
  \end{eqnarray*}
  Rather than computing $X^\rho_\circ \cap X_\pi$ down in $GL(k+n)/B$,
  we will compute up in $GL(k+n)$ and project. So we intersect $C$ 
  (which maps isomorphically to its projection in $GL(n+k)/B$) and $\barX_\pi$ 
  (which inside $GL(k+n)$, is a union of fibers of the projection):
  \begin{eqnarray*}
  C \cap \barX_\pi 
  &=&
  \left\{ 
    M = \left[
      \begin{array}{cc}  N \in M_{k\times n} & Id_k \\  Id_n & 0  \end{array}
    \right] : 
    \forall 0 \leq i < j\leq n,
    rank(M_{\leq k+i,\leq j}) \leq i + r_{i+1,j}
  \right\} \\
  &=&
  \left\{ 
    M = \left[
      \begin{array}{cc}  N \in M_{k\times n} & Id_k \\  Id_n & 0  \end{array}
    \right] : 
    \forall 0 \leq i < j\leq n,
    rank(N_{[i+1,j]}) \leq r_{i+1,j}
  \right\} \\
  &\iso&
  \left\{ 
    N \in M_{k\times n} : 
    \forall 1 \leq i' \leq j\leq n,
    rank(N_{[i',j]}) \leq r_{i',j}
  \right\} \qquad\qquad (i'=i+1) \\
  &=& \barPi_r.
  \end{eqnarray*}
  Since $C$ is projecting isomorphically to $X^\rho_\circ$, 
  this intersection is projecting isomorphically to $X^\rho_{\pi\circ}$.
\end{proof}

Our running example will be the following $r$ on the left,
giving the $\pi$ on the right, 
$$ r = \left[
    \begin{array}{ccccc}
      0 & 0 & 0 & 0 & 1 \\
        & 0 & 0 & 0 & 0\\
        &   & 1 & 0 & 0\\
        &   &   & 0 & 0\\
        &   &   &   & 0\\
    \end{array}    \right], \qquad
  \pi = \left[
    \begin{array}{cccccccc}
      1 & 0 & 0 & 0 & 0 & & & \\
      0 & 1 & 0 & 0 & 0 & & & \\
      0 & 0 & 0 & 1 & 0 & & &
      \\
      0 & 0 & 0 & 0 & 1 & 0 & 0 & 0 \\
        & 0 & 0 & 0 & 0 & 1 & 0 & 0 \\
        &   & 1 & 0 & 0 & 0 & 0 & 0 \\
        &   &   & 0 & 0 & 0 & 1 & 0 \\
        &   &   &   & 0 & 0 & 0 & 1
    \end{array}    \right].
$$

\begin{Corollary}[of the proof]
  \label{cor:coveringrelations}
  Partially order the set of interval rank matrices by 
  $r \leq r'$ if $\barPi_r \supseteq \barPi_{r'}$ (the reversal is
  to match Bruhat order). Then $r \leq r'$ is a covering relation
  iff one of the following possibilities holds:
  \begin{enumerate}
  \item $J(r)$ and $J(r')$ agree, except on a rectangle in which 
    $J(r)$ has $1$s only in the NW and SE corners, whereas
    $J(r')$ has $1$s only in the SW and NE corners.
  \item $J(r)$ and $J(r')$ agree, except that a $1$ in $J(r)$ has
    moved one column to the left (into a column that was previously zero),
    or one row down (into a row that was previously zero).
  \end{enumerate}
\end{Corollary}

\begin{proof}
  The covering relations in $S_n$ Bruhat order, when expressed in terms
  of permutation matrices, are exactly as described in (1).
  When we embed $J(r),J(r')$ into $(n+k)\times (k+n)$ permutation
  matrices as in the proof of theorem \ref{thm:rankmatrices},
  we acquire more $1$s in the permutation matrix, hidden in the
  rectangles $A_1$ and $A_2$ (and $A_1',A_2'$). 
  The covering relations of type (2) are the ones that involve
  moving these hidden $1$s.
  (The covering relations involving them are rather limited by the fact
  that the $1$s in those rectangles are arranged NW/SE.)
\end{proof}

We denote covering relations by $r \lessdot r'$.

\begin{Lemma}\label{lem:intersectIRvars}
  The intersection of two interval rank varieties is a reduced union
  of other interval rank varieties. 
  The same follows for their positroid varieties inside $\Grkn$.
\end{Lemma}

\begin{proof}
  The intersection $\barX_w \cap \barX_v$ of two matrix Schubert
  varieties is a reduced union of other matrix Schubert varieties
  (by \cite[proof of lemma 3.11 after lemma 6.1]{Fulton92}
  and \cite[theorem 3]{RamanathanACM}), namely those $\barX_u$ 
  where $u$ is a least upper bound in $S_n$ of $w,v$.
  Let $\pi_1,\pi_2$ be the $(k+n)\times(n+k)$ matrices associated in
  the proof of theorem \ref{thm:rankmatrices} to two interval rank varieties.

  If neither $w$ nor $v$ have a descent between positions $i,i+1$,
  then each $u$ won't either. (Proof: the descent condition says that
  the corresponding Schubert varieties $X_w,X_v \subseteq GL_n/B$ are
  unions of fibers of the map $GL_n/B \onto GL_n/P_i$, hence their
  intersection is too, hence any component of it is too.) 
  Using transpose, the same holds for $w^{-1},v^{-1}$.
  So for each component $\barX_\rho$ of $\barX_{\pi_1} \cap \barX_{\pi_2}$,
  the permutation $\rho$ has no descents in positions $1\ldots k$,
  and $\rho^{-1}$ has none in positions $n+1\ldots n+k$.
  Thus it is necessarily of the same form as the $\pi$ from 
  the proof of theorem \ref{thm:rankmatrices}, 
  and hence $C \cap \barX_\rho$ is again an interval rank variety.

  To compute the corresponding intersection inside $\Grkn$, 
  we intersect with the copy of the Stiefel manifold inside $C$.
  (This drops those components with $s>0$, for the $s$ 
  from the proof of theorem \ref{thm:rankmatrices}.)
\end{proof}

We now define the ``essential set'' for an interval rank matrix, an analogue
of Fulton's essential set for a Northwest rank matrix (used to define
matrix Schubert varieties). First draw lines strictly to the South,
and strictly to the West, of each $1$ in $J(r)$, which we think of as 
crossing out boxes. Then also cross out any empty row or column (with no $1$).
Call the remaining matrix entries (which {\em includes} all the $1$s)
the \defn{strict S/W diagram} of $J(r)$. Define the \defn{essential set} to be 
the Northeast corners of the strict S/W diagram. (Fulton's essential set, on
his different sort of rank matrix, is the SE corners of the weak S/E diagram.)

In the running example above, 
the essential rank conditions are $r_{33} \leq 0$, $r_{15} \leq 3$.

\begin{Proposition}\label{prop:essset}
  The interval rank variety 
  $ \barPi_r = \{ N \in M_{k\times n} : r(N) \leq r\text{ entrywise} \} $
  is defined as a scheme already by the rank conditions 
  $r(N)_{ij} \leq r_{ij}$ for essential boxes $(i,j)$ in $J(r)$'s
  strict S/W diagram.
\end{Proposition}

\begin{proof}
  First, let $(i,j)$ be a matrix entry not lying in the strict S/W diagram
  at all.
  Thus $(i,j)$ is crossed out, say from the North,
  either by a $1$ to the North at $(i'<i, j)$
  or because it is in an empty column.
  (The cases of being crossed out from the East will work the same way.)
  Therefore there is no $1$ lying weakly South of $(i,j)$. 
  Hence $r_{ij} = r_{i\ j-1} + 1$, so the $(i,j)$ rank condition
  is implied by the $(i,j-1)$ rank condition.

  Now assume $(i,j)$ is in the strict S/W diagram, but is not a NE corner.
  Then there is another diagram box at $(i-1,j)$ or $(i,j+1)$; we treat
  the first case. Since $(i-1,j)$ is not crossed out, it {\em has} a $1$ weakly
  to its West, at some $(i-1,j'<j)$. 
  Hence $r_{ij} = r_{i-1\ j}$, so the $(i,j)$ rank condition
  is implied by the $(i-1,j)$ rank condition.

  This lets us trace each rank condition $(i,j)$ outside the
  ``essential set'' to another rank condition with the same rank bound
  to the North or East, or, to one with a lower rank bound to the South or West.
  Clearly this process must terminate, at an essential box. 
  So the rank conditions from the essential boxes imply all the others.
\end{proof}

The term ``essential'', taken from \cite{Fulton92}, is misleading; 
if $r_{i,k} = r_{i,j} + r_{j+1,k}$, and the latter two define
essential rank conditions, the $r_{i,k}$ condition is certainly implied
but may also be ``essential'', as occurs in the example in \S \ref{sec:ex}.
(This phenomenon does not occur in Fulton's context 
\cite[lemma 3.14]{Fulton92}.) 

Given a subvariety $Y\subseteq \Grkn$, let $\lambda(Y)$ denote
the maximum $\lambda$ such that $Y \subseteq X_\lambda$.
(Proof of existence: if $Y \subseteq X_\lambda$ and $Y \subseteq X_\mu$,
then since $Y$ is irreducible it is contained in one of the components
$X_\nu$ of $X_\lambda \cap X_\mu$, and $\nu \geq \lambda,\mu$.) 
Similarly, let $\mu(Y)$ denote
the minimum $\mu$ such that $Y \subseteq X^\mu_\circ$.
If $Y$ is $T$-invariant, then $\lambda(Y),\mu(Y)$ are the minimum and
maximum of $Y^T$ in Bruhat order.
Call $X_{\lambda(Y)}^{\mu(Y)}$ the \defn{Richardson envelope} of $Y$;
it is the unique smallest Richardson variety containing $Y$.

\begin{Proposition}\label{prop:envelope}
  Let $r$ be an interval rank matrix of size $n$ with $r_{1n} = k$.
  Then $\lambda(\Pi_r)$ has its $1$s in the empty rows of $J(r)$,
  and $\mu(\Pi_r)$ has its $1$s in the empty columns of $J(r)$.
  The codimension of $\Pi_r$ inside its Richardson envelope
  $X_{\lambda(\Pi_r)}^{\mu(\Pi_r)}$ is the number of pairs of $1$s in $J(r)$
  arranged NE/SW.  
\end{Proposition}

In our running example, this number of pairs is $1$.
The Richardson envelope $X_{01011}^{11010}$ is defined by 
$r_{13} \leq 2$, $r_{35} \leq 2$, $r_{15} \leq 3$. 
In that larger variety (or really its Stiefel cone in matrix space), 
the middle column is a vector contained in the
$2$-plane spanned by the left three columns, intersect the $2$-plane
spanned by the right three columns, inside the ambient $3$-space,
making that column unique up to scale. 
Hence imposing $r_{33}=0$, that that column
be the zero vector, only drops the dimension by $1$, which is the
computed codimension in the Richardson envelope.

\begin{proof}
  The minimum $\lambda(\Pi_r)$ is determined by the
  rank conditions $\barX_r$ satisfies on its initial intervals $\{[1,j]\}$,
  i.e. the first row $(r_{1j})$ of $r$. That, in turn, is determined
  by the empty columns of $J(r)$. The same analysis connects $\mu(\Pi_r)$
  to the last column $(r_{jn})$ of $r$ to the empty rows of $J(r)$. 

  Then, using theorem \ref{thm:rankmatrices}, 
  $$ \dim \Pi_r = \dim \barX_r - \dim GL(k) = \dim X^\rho_{\pi\circ} - k^2
  = \ell(\rho) - \ell(\pi) - k^2 
  = k(n-k) - \ell(\pi) $$
  where (as in its proof)
  $$ \pi = \left[
    \begin{array}{cc}
      A_1 & 0 \\
      J & A_2 
    \end{array}    \right], \qquad
  \rho = 
  n\!+\!1\ n\!+\!2\ \ldots n\!+\!k\ 1\ 2\ \ldots\ n.
  $$
  To determine $\ell(\pi)$, we count inversions, i.e. pairs of $1$s
  in $\pi$ aligned NE/SW rather than NW/SE. These pairs come in
  three types: one in the $A_1$ block and one in the $J$ block,
  one in the $A_2$ block and one in the $J$ block,
  or both in the $J$ block.

  Each such pair with one $1$ in the $A_1$ block and 
  one $1$ in the $J$ block corresponds to a $1$ in $\mu(\Pi_r)$ occurring
  before a $0$, so the number of them is 
  $\codim \left(X^{\mu(\Pi_r)} \subseteq \Grkn\right)$.
  Similarly,
  the number of such pairs with one $1$ in the $A_2$ block and one $1$ 
  in the $J$ block is $\codim \left(X_{\lambda(\Pi_r)} \subseteq \Grkn\right)$. 
  Hence 
  $\ell(\pi) 
  = \codim \left(X_{\lambda(\Pi_r)}^{\mu(\Pi_r)} \subseteq \Grkn\right) + c$,
  where $c$ is the number of NE/SW pairs in $J(r)$. Finally, 
  \begin{eqnarray*}
    \codim \left(\Pi_r \subseteq X_{\lambda(\Pi_r)}^{\mu(\Pi_r)}\right)
    &=& \dim X_{\lambda(\Pi_r)}^{\mu(\Pi_r)} - \dim \Pi_r \\
    &=& \dim \Grkn 
    - \codim \left(X_{\lambda(\Pi_r)}^{\mu(\Pi_r)} \subseteq \Grkn\right)
    - k(n-k) + \ell(\pi) 
    \\     &=& c.
  \end{eqnarray*}
  \vskip -.3in
\end{proof}


For later use, we will want some handle on the $T$-fixed points
$\left(\Pi_r\right)^T \subseteq \Grkn^T 
= \{k$-dimensional coordinate subspaces$\}
\iso \{$words $\lambda$ of length $n$ with $n-k$ $0$s and $k$ $1$s$\}$.

\begin{Lemma}\label{lem:matchings}
  Let $r$ be an interval rank matrix, let $\lambda$ be
  a word of length $n$ with $n-k$ $0$s and $k$ $1$s, and
  $V_\lambda$ the corresponding $k$-dimensional coordinate subspace.
  \begin{enumerate}
  \item If for some $(i,j)$, $r_{ij} < \sum_{k\in [i,j]} \lambda_k$,
    then $V_\lambda \notin \Pi_r$.
  \item Let $supp(J(r))$ denote the locations of the $1$s in the partial 
    permutation matrix corresponding to $r$, 
    and $m : supp(J(r)) \to \{1,\ldots,n\}$ an injection such that
    $m\big((i,j)\big) \in \{i,\ldots,j\}$. Let $\lambda$ be $0$ on the
    image of $m$ (for \defn{matching}), and $1$ on the complement. 
    Then $V_\lambda \in \Pi_r$.
  \item (Hall's marriage theorem for interval rank varieties)
    If $V_\lambda \in \Pi_r$, then there exists a matching $m$ 
    as described above.
  \item Let $r' \gtrdot r$ be a covering relation 
    as in corollary \ref{cor:coveringrelations}, and $d$ be the position
    of the unique $1$ (if type (2)) or the Southwestern of the two $1$s
    (if type (1)) in $supp(J(r'))\setminus supp(J(r))$. 
    Let $m$ be a matching of $r$, and think of it as a map from $supp(J(r))$
    to the diagonal of $r$, whose image has complement $\lambda$.

    If any $1$ in a row above $d$ is matched to an entry in a row above $d$,
    and any $1$ in a column to the right of $d$ is matched to an entry 
    in a column to the right of $d$, 
    then $V_\lambda \in \Pi_r \setminus \Pi_{r'}$.
\junk{

    More specifically, define $m_{left},m_{right}: supp(J(r))\to \{1,\ldots,n\}$
    by
    $$ m_{left}(e) = 
    \begin{cases}
      i(e) & \text{if } j(e)<j(d) \\
      j(e) & \text{if } j(e)\geq j(d),
    \end{cases}
    \qquad\qquad
    m_{right}(e) = 
    \begin{cases}
      i(e) & \text{if } i(e)<i(d) \\
      j(e) & \text{if } i(e)\geq i(d).
    \end{cases}
    $$
    Then $m_{left},m_{right}$ are matchings of this sort, 
    and the complements $\lambda_{left},\lambda_{right}$ of their images
    have $V_{\lambda_{left}},V_{\lambda_{right}} \in \Pi_r \setminus \Pi_{r'}$.
}
  \end{enumerate}
\end{Lemma}

\begin{proof}
  Let $M_\lambda$ be the $k\times n$ matrix with the $k\times k$ 
  identity matrix in the $k$ columns $\{i :\lambda_i=1\}$, and other
  columns $0$. Then $V_\lambda$ is the row-span of $M_\lambda$,
  so $V_\lambda \in \Pi_r$ iff $M_\lambda \in \barPi_r$.
  Then since $r(M_\lambda)_{ij} = \sum_{k\in [i,j]} \lambda_k$,
  we have $M_\lambda \in \barPi_r$ iff 
  $r_{ij} \geq \sum_{k\in [i,j]} \lambda_k$ for all $i\leq j$.
  \begin{enumerate}
  \item If some $r_{ij} < 
    r(M_\lambda)_{ij}$, 
    then $M_\lambda \notin \barPi_r$, hence $V_\lambda \notin \Pi_r$.
  \item For each $i\leq j$, 
    \begin{eqnarray*}
       r(M)_{ij} 
    = \sum_{k\in [i,j]} \lambda_k 
    &=& \big| [i,j] \setminus m(supp(J(r))) \big| \\
    &\leq& \big| [i,j] 
    \setminus m\left(supp(J(r)) \text{ southwest of }(i,j) \right) \big| \\
    &=& \big| [i,j] \big|
    - \big| supp(J(r)) \text{ southwest of }(i,j) \big| 
    = r_{ij}.
    \end{eqnarray*}
  \end{enumerate}
  Part (3) is trivial if $J(r)$ is diagonal, and $m$ is
  the identity map, if we think of its target set as the diagonal entries
  $\{(i,i)\}$. We use this as the base case for an induction.

  The induction step is to prove that if $r' \gtrdot r$ is a covering relation 
  of type (1) or (2) in the sense of corollary \ref{cor:coveringrelations}, 
  then the set of coordinate subspaces in $\Pi_{r'}$ with matchings is
  contained in the set of coordinate subspaces in $\Pi_{r}$ with matchings.

  If $r' \gtrdot r$ is a type (2) covering relation, one $1$ in $J(r')$
  moves one step North or East to give $J(r)$, giving a simple correspondence
  $c$ between their sets of $1$s. If $m'$ is a matching for $r'$, 
  then $m := m' \circ c$ is a matching for $r$; the $1$ in $J(r)$ first
  maps South or West one step via $c$, then maps SW via $m'$.

  If $r' \gtrdot r$ is a type (1) covering relation, then two $1$s in $J(r')$
  at positions $p$-NE-of-$q$ move to give those in $J(r)$, 
  at positions $s$-NW-of-$t$. 
  To define a correspondence $c$ as in the type (2) case, 
  and take $m := m' \circ c$, 
  we need to choose $p\mapsto s, q\mapsto t$ or $p\mapsto t, q\mapsto s$.
  There are three possibilities:
  \begin{itemize}
  \item $m'(p)$ is SW of $q$, hence SW of both $s$ and $t$. 
    Then either possibility for $c$ produces a suitable $m$.
  \item $m'(p)$ is SW of $s$, but not SW of $t$. 
    Then $c$ should take $p\mapsto s, q\mapsto t$.
  \item $m'(p)$ is SW of $t$, but not SW of $s$. 
    Then $c$ should take $p\mapsto t, q\mapsto s$.
  \end{itemize}
  For example, if $r'$ is the $5\times 5$ matrix in our running example 
  and $m'$ maps each both $p$ and $q$ due West (so $\lambda = 01011$), 
  then we face the second possibility. 
  Whereas if $m'$ maps each due West (so $\lambda = 11010$), we face the third.
  The first possibility cannot occur for this $r'$ since $m'(q)$ is forced and
  $m'$ is injective.

  For part (4), let $i,j$ be the row and column of $d$. 
  By the condition on $m$, $r(M_\lambda)_{ij} = r_{ij}$. 
  But by the condition on $d$, $r'_{ij} = r_{ij} - 1$. 
  So $r(M_\lambda)_{ij} > r'_{ij}$, hence $V_\lambda \notin \Pi_{r'}$.
\end{proof}


As we explained in the introduction, the varieties $\Pi_r$ are a
special class of ``positroid varieties'' \cite{KLS}, which in their
full generality allow for rank conditions on {\em cyclic} intervals.
There is another connection between 
the $\{\barPi_r\}$ and positroid varieties (much like the double connection 
between matrix Schubert varieties and Schubert varieties), as follows.
Embed $\barPi_r \times M_{k\times k}$ into $M_{k\times(n+k)}$ by imposing no 
rank conditions on the last $k$ columns; 
this is $\barPi_{r'}$ for a suitable $r'$.
Then $\barPi_r$ is isomorphic to the affine open set on the positroid variety
$GL(k) \dom (\barPi_{r'} \cap \St_{k,n+k})$ where one asks that the
last $k$ columns are linearly independent. In her thesis \cite{Snider}
Snider shows more generally that the natural affine patches on arbitrary 
positroid varieties are isomorphic to certain Kazhdan-Lusztig varieties 
in the {\em affine} flag manifold.

\section{The Vakil variety of a puzzle path}\label{sec:variety}

Let $\gamma$ be a (labeled) puzzle path, as defined in \S\ref{ssec:puzzleintro}.
In this section we will use $\gamma$ to single out $n-k$ horizontal edges
in the puzzle triangle (possibly along the bottom), and use them to
define interval rank conditions. 

There are a number of conditions that we require the labels on $\gamma$
to satisfy, all of which are implied by ``there should be a way to successively
add puzzle pieces to $\gamma$, culminating in a final path''. 
These conditions are:
\begin{itemize}
\item On the boundary of the puzzle, there are only $0$s and $1$s, no $R$s
  or $K$s. (In particular, on initial or final paths 
  there are only $0$s and $1$s, no $R$s.)
\item The only place a $K$ may appear is on the kink, so $\bslash K$.
\item Say the first $i$ steps of $\gamma$ are SE. 
  The number of $\bslash 0$s on those edges should be at least the
  number of $\dash 0$s in the last $i$ steps (necessarily all West).
\item The number of $\bslash 0$s 
  should equal the number of $\fslash R$s plus the number of $\dash 0$s.
  (In particular, on final paths as on initial paths,
  there are only $0$s and $1$s, no $R$s or $K$s.)
\item If the kink is $\bslash R$ or $\bslash K$, 
  after it there must be a $\fslash 1$ or $\dash 1$ 
  before there is any $\fslash R$ or $\dash 0$. 
\item If the kink is $\bslash 0$ or $\bslash K$, 
  after it there must be an $\fslash R$ or $\dash 0$ 
  before there is any $\dash 1$.
\item If the kink is $\bslash K$,
  after it there must be a $\fslash 1$ 
  before any $\fslash R$ or $\dash 0$, in turn
  before any $\dash 1$ 
  (as implied by the previous two conditions).
\end{itemize}

\subsection{Pink rays and pink dots}

We first draw $n-k$ \emph{pink rays} aligned NE/SW, and $n-k$ more
aligned NW/SE. At certain crossing points of these rays, we will place $n-k$
\emph{pink dots}. These define a partial permutation as in theorem
\ref{thm:rankmatrices}, giving rank conditions
$$ rank(M_{[i,j]}) 
\ \leq\  
|[i,j]| - \#\{\text{pink dots on edges $e$ with }i\leq i(e) \leq j(e) \leq j\}$$
where $(i(e),j(e))$ were defined in \S \ref{ssec:puzzleintro}.

\begin{figure}[htbp]
  \centering
  \epsfig{file=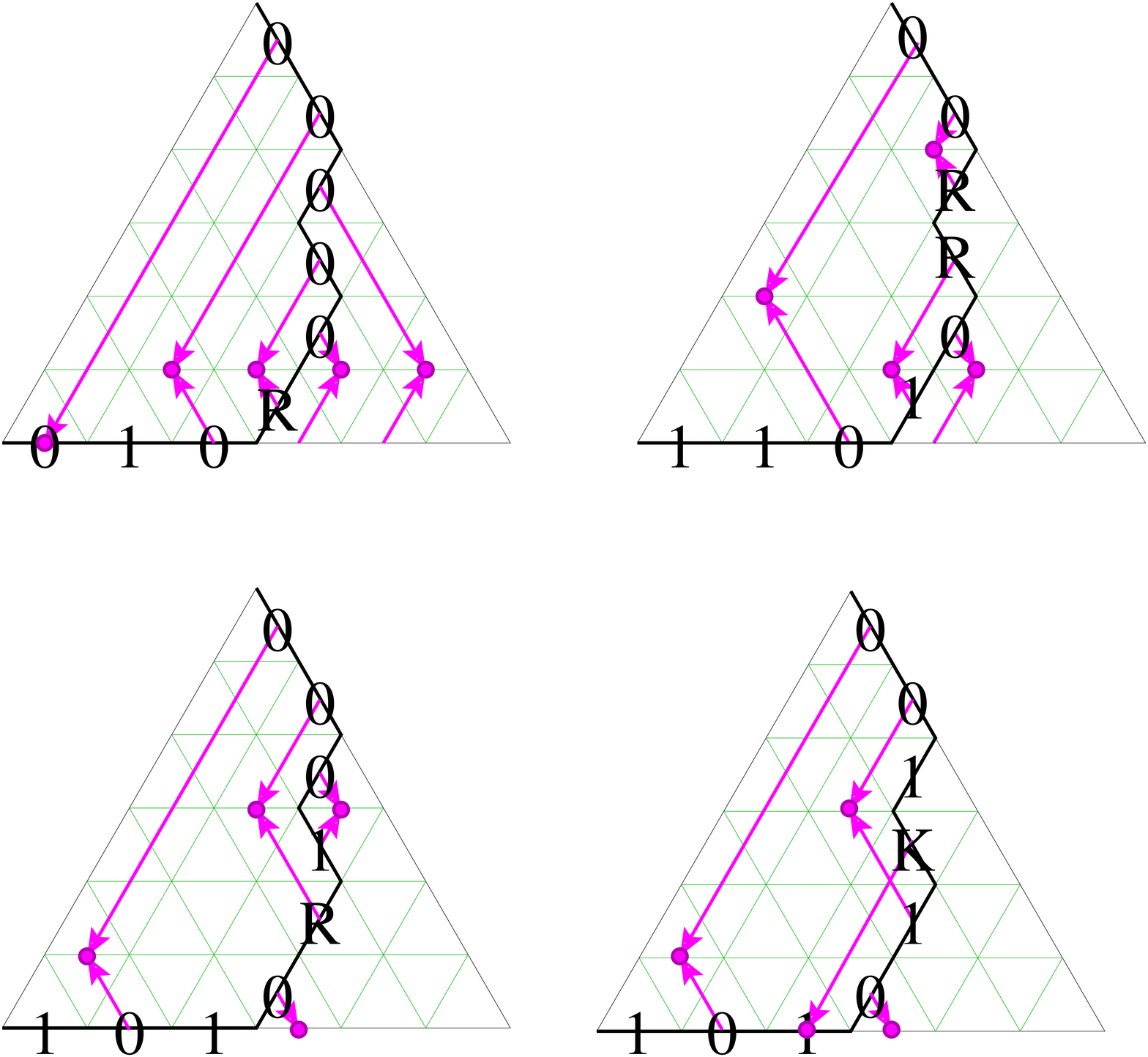,width=5in}
  \caption{Four labelings on the same unlabeled puzzle path, with their
    associated pink rays colliding at pink dots.}
  \label{fig:puzzlepathexs}
\end{figure}

Each pink ray emanates from the midpoint of an edge, so does {\em not}
follow puzzle edges (rather, it is only parallel to them). 
Cut the puzzle triangle into a left half and right half along $\gamma$.
On the left side of $\gamma$:
\begin{itemize}
\item The rays are aligned NW \& SW.
\item Each $\bslash 0$ (which may include the kink) 
  has a SW-pointing ray.
\item Each $\dash 0$ and $\fslash R$ has a NW-pointing ray.
\item If the kink is $\bslash R$ or $\bslash K$, it gets a SW-pointing ray,
  and causes the next $\fslash 1$ South of it to get a NW-pointing ray.
\end{itemize}
On the right of $\gamma$:
\begin{itemize}
\item The rays are aligned NE \& SE. 
\item Each $\fslash 0$ has a SE-pointing ray.
\item If the kink is $\bslash 1$, and there is a
  $\fslash 0$ somewhere above it, the $\bslash 1$ gets a NE-pointing ray.
\item Immediately to the right of the rightmost edge of $\gamma$ on
  the bottom of the triangle, enough other NE-pointing rays are placed
  to match the number of SE-pointing rays. These are the only rays 
  that come out of puzzle edges not in $\gamma$.
\end{itemize}

We will not extend these rays forever, but only to certain crossings,
which we will label with ``pink dots''.

\begin{Theorem}\label{thm:pinkdots}
  Let $\gamma$ be a puzzle path, with pink rays attached as described above.
  Assume first that the kink is not labeled $\bslash K$.

  Then there is a unique way to pair up the pink rays, such that each 
  pair of paired rays are extended to intersect at a \defn{pink dot} (on a 
  horizontal edge), and there are no other intersections of pink rays.

  If the kink is labeled $\bslash K$, then almost the same is true:
  there is one crossing (with no pink dot), 
  of the pink ray coming SW out of the $\bslash K$ and 
  the one coming NW out of the next $\fslash 1$.

  No pink dot is in the same column (NW/SE or NE/SW) as another.
\end{Theorem}

\begin{proof}
  Throughout this proof we use the conditions on the labeling of $\gamma$,
  generally without comment.

  We worry first about matching the ray out of the kink (if any).
  \begin{itemize}
  \item If the kink is a $\bslash 0$ or $\bslash R$, 
    it has a pink ray going SW. Further South along $\gamma$, 
    there must be a positive number of pink rays going NW 
    (either from the first $\fslash 1$ after the $\bslash R$, 
    or from an $\fslash R$ or $\dash 0$). The first such NW ray 
    below the kink {\em must} meet the SW ray out of the kink -- 
    there are no other rays beforehand for the SW ray to collide with.
    So declare those two rays matched up, place a pink dot where they cross,
    and extend them no further.
  \item If the kink is a $\bslash K$, then further South along $\gamma$, 
    there must be at least two pink rays going NW (by the last condition
    on puzzle paths).
    The SW ray from the $\bslash K$ is required to cross through 
    the first NW ray, and as above must be matched up with the second NW ray.
  \item If the kink is a $\bslash 1$, and there is no $\fslash 0$ above it,
    then there is no pink ray out of the $\bslash 1$ to consider.
  \item If the kink is a $\bslash 1$, and there is a $\fslash 0$ above it,
    then there is a pink ray NE out of the $\bslash 1$
    which must be matched up with the SE ray out of the $\fslash 0$ 
    most closely above it.
  \end{itemize}

  Now we match up the remaining pink rays on the left half of the puzzle
  triangle (as cleft by $\gamma$). The SW rays come from $\bslash 0$s
  on the NE side of the puzzle triangle, and the NW ones come from $\fslash R$s
  along $\gamma$ and $\dash 0$s on the South side of the puzzle triangle. 
  By assumption on $\gamma$, there are the same number of these rays
  (which involves a small case check over the possibilities for the kink).
  To avoid creating crossings, the SW rays must be matched with the NW rays
  in order, giving the uniqueness. 
  To ensure that the $k$th SW ray crosses the $k$th NW ray at all
  (each extended infinitely), we use the first condition on $\gamma$.

  Wholly independently,
  we match up the remaining pink rays on the right half of the puzzle
  triangle. Here the number of NE rays from the bottom edge of the puzzle 
  triangle ({\em not} on $\gamma$) was chosen to match the number of SE rays 
  from $\fslash 0$s along $\gamma$. Again, to avoid creating crossings, 
  the SW rays must be matched with the NW rays in order, giving the uniqueness.

  We must check that no two pink dots are in the same NE/SW or NW/SE column.
  Group the dots into Left, Kink, and Right according to their NE/SW column.
  Obviously two pink dots in different groups cannot be 
  in the same NE/SW column,
  and it is easy to see also that two Left pink dots cannot be in
  the same NE/SW or NW/SE column, nor can two Right pink dots.
  There is at most one pink dot in the Kink group.
  
  It remains to show that no two pink dots in different groups can be
  in the same NW/SE column. Drop the Kink dot (if any) into the
  Left group or Right group depending on which side of $\gamma$ it lies on.
  If a Left dot were NW of a Right dot (SE being obviously impossible), 
  the NW ray pointing to the Left dot 
  would emanate from the same $\gamma$-edge
  as the SE pointing to the Right dot, a contradiction. 
\end{proof}

\subsection{The Vakil variety}\label{ssec:vakilvariety}

Let $\barPi_\gamma \subseteq M_{k\times n}$
denote the interval rank variety and 
$\Pi_\gamma \subseteq \Grkn$ its associated Vakil subvariety of the 
Grassmannian, using the rank conditions from the $n-k$ pink dots placed
according to theorem \ref{thm:pinkdots}. 
It is easy to carry over the definition of ``essential set'' 
from proposition \ref{prop:essset} to puzzle paths and their pink dots:
cross out all NW/SE columns and NE/SW columns with no pink dots, and
strictly SW,SE of each pink dot. Then the essential rank conditions
correspond to the locally Northernmost horizontal edges remaining,
which we will call \defn{essential edges}.

\begin{Proposition}\label{prop:gammacodim}
  If $\gamma$ has no kink, it is easy to compute the codimension
  of $X_\gamma$ inside its Richardson envelope: it is the number of
  pairs ``$\fslash R$ above $\fslash 0$'' occurring along $\gamma$.

  If $\gamma$ has a kink, we must add correction terms:
  \begin{itemize}
  \item If the kink is $\bslash 1$, add the number of $\fslash R$s above the
    last $\fslash 0$ above the kink, and if there is a $\fslash 0$
    above the $\bslash 1$, add also the number of $\fslash 0$s
    below the $\bslash 1$.
  \item If the kink is $\bslash 0$, \\
    add the number of $\fslash R$s above the kink, 
    plus the number of $\fslash 0$s 
    below the first $\fslash R$ below the $\bslash 0$.
  \item If the kink is $\bslash R$, \\
    add the number of $\fslash R$s above the kink, 
    plus the number of $\fslash 0$s 
    below the first $\fslash 1$ below the $\bslash R$.
  \item If the kink is $\bslash K$, \\
    add the number of $\fslash R$s above the kink, 
    plus the number of $\fslash 0$s 
    below the first $\fslash 1$ below the $\bslash K$.
  \end{itemize}
\end{Proposition}

\begin{proof}
  First consider the case of no kink.
  The noncrossing condition on the pink rays implies that two pink dots
  both on the left, or both on the right, of $\gamma$ will not contribute 
  to the codimension (as computed in proposition \ref{prop:envelope}).
  A pink dot on the left is NW of a $\fslash R$, and a pink dot on
  the right is SE of a $\fslash 0$; such a pair only contributes if the
  $\fslash R$ occurs above the $\fslash 0$.

  If there is a kink, split the pink dots into three classes: 
  \begin{enumerate}
  \item those in the NE/SW columns to the left of $\gamma$,
  \item the at most one pink dot in the NE/SW column of the kink, and
  \item those in the NE/SW columns to the right of $\gamma$.
  \end{enumerate}
  Again, there can be no contribution from pairs of pink dots in the
  same class. Case-by-case analysis comparing groups 1 and 2, 1 and 3, 2 and 3
  gives the rest.
\end{proof}

\begin{Proposition}\label{prop:initialpath}
  Let $\gamma$ be the initial path, with a labeling, $\mu$ on the NE side
  of the puzzle triangle and $\nu$ along the bottom edge (both read
  left to right). Then $X_\gamma$ is the Richardson variety $X^\nu_\mu$.
\end{Proposition}

\begin{proof}
  Proposition \ref{prop:gammacodim} easily gives that $X_\gamma$
  is codimension $0$ in its Richardson envelope, so we merely have
  to determine that envelope.

  Each pink dot is NW of a $\dash 0$ and SW of a $\bslash 0$.
  With this, we can determine $(r_{1j})$ and $(r_{in})$, 
  hence the Richardson envelope $X^\nu_\mu$.
\end{proof}

\begin{Proposition}\label{prop:finalpath}
  Let $\gamma$ be the final path, with a labeling $\lambda$ on the NW side
  of the puzzle triangle (read left to right). 
  Then $X_\gamma$ is the opposite Schubert variety $X^\lambda$.
\end{Proposition}

\begin{proof}
  Again, proposition \ref{prop:gammacodim} gives that $X_\gamma$
  is codimension $0$ in its Richardson envelope.
  Let $\gamma'$ be the initial path with labels $0^{n-k} 1^k$ on the NE side
  and $\lambda$ on the S side. 
  It is easy to check that the pink dots for $\gamma$ are in the same locations
  as for $\gamma'$. Now apply proposition \ref{prop:initialpath}.

  Alternately, apply proposition \ref{prop:essset} to see that the
  only essential rank conditions are from $(r_{1j})$, and check that
  those define $X^\lambda$.
\end{proof}

The next proposition is crucial to the puzzle combinatorics: it will say 
that a Vakil variety associated to a puzzle path is $\shift$-invariant 
unless there are multiple ways to fill in the next puzzle piece.

\begin{Proposition}
  Let $\gamma$ be a non-final puzzle path, 
  where the next rhombus to be filled has NE/SW column $i$, NW/SE column $j$. 
  The essential edges occurring on the right-hand side of $\gamma$ only occur in
  the $i$th or $(i+1)$st NE/SW column. 

  If the kink is $\bslash 1$ and the next edge is $\fslash 0$,
  there is an essential edge $e$ at $i(e) = i+1$, $j(e) \geq j$. 
  Otherwise no essential edges $e$ have $i \notin [i(e),j(e)] \ni j$.
\end{Proposition}

\begin{proof}
  If an edge $e$ in the $k$th NE/SW column, $k>i+1$, is not crossed out 
  (as described at the beginning of \S \ref{ssec:vakilvariety}),
  we claim the horizontal edge $e'$ just NW of $e$ is also not crossed out.
  Proof: since $e$ is not crossed out from the NW, $e'$ is also
  not crossed out from the NW. Since $e$ is not crossed out from the NE, 
  there is a pink dot in its NE/SW column, weakly SW of it. 
  By the way we placed pink dots on the right side of $\gamma$, 
  there is also a pink dot strictly SW of $e'$. 
  Hence $e'$ is also not crossed out from the NW, so not at all.
  Since $e'$ is not crossed out, $e$ is not essential.

  Now consider essential edges $e$ with $i < i(e) \leq j \leq j(e)$.
%
%
  By the previous paragraph, $i(e)=i+1$. Since the horizontal edge $e'$ 
  just NW of $e$ is crossed out, necessarily from its NE, either there is 
  no pink dot in the kink column or the only pink dot is strictly NE of $e'$.
  Either way, the kink must be $\bslash 1$, and every edge SW of the
  kink is crossed out.

  For $e$ to not be crossed out, there must be a pink dot weakly SE of it,
  so there must be a $\fslash 0$ NW of it. If $i(e)=j$, that $\fslash 0$
  is the next edge below the kink $\bslash 1$, which was the possibility
  singled out. In this case the horizontal edge just SE of that $\fslash 0$,
  at $(i+1,j)$, is not crossed out, but the edge just NW of it, at $(i,j)$, is.
  Hence some edge weakly NE of $(i+1,j)$ is essential.

  The remaining case is $j>i(e)$. 
  Then we have located some $\fslash 0$ above the $\bslash 1$, so there is
  a pink ray NW from the $\bslash 1$ meeting the first $\fslash 0$ above it.
  (In particular, there is a pink dot in the kink column.)
  But for $e$ to be essential, the horizontal edge just NW of $e$
  must be crossed out by a pink dot strictly to its NE, which doesn't
  fit with that dot being SE of the {\em first} $\fslash 0$ above
  the $\bslash 1$. So $e$ cannot be essential.
\end{proof}

\section{A detailed example: 
  $\mu=0101$, $\nu=1010$}\label{sec:ex}

We follow the Vakil degeneration of the Richardson variety 
$Y_1 := X_{0101}^{1010}$, including the sweeps and the intersections of
components as in theorem \ref{thm:geom}. 
By proposition \ref{prop:initialpath}, $Y_1$ is associated to 
an initial puzzle path labeled with $0101$ on NE, $1010$ on S.
As we will prove in general in \S \ref{sec:filling},
each shift/sweep operation will correspond to moving this path leftward
by adding a rhombus puzzle piece.

Initially, $Y_1$'s essential rank conditions are $r_{12} , r_{34} \leq 1$. 
There is one more ``essential edge'' for the corresponding $J$, but
that rank condition $r_{14} \leq 2$ is the direct sum of these two.
The first shift, $3\to 4$, does nothing to the rank conditions 
or to the placement of the pink dots:

\centerline{\epsfig{file=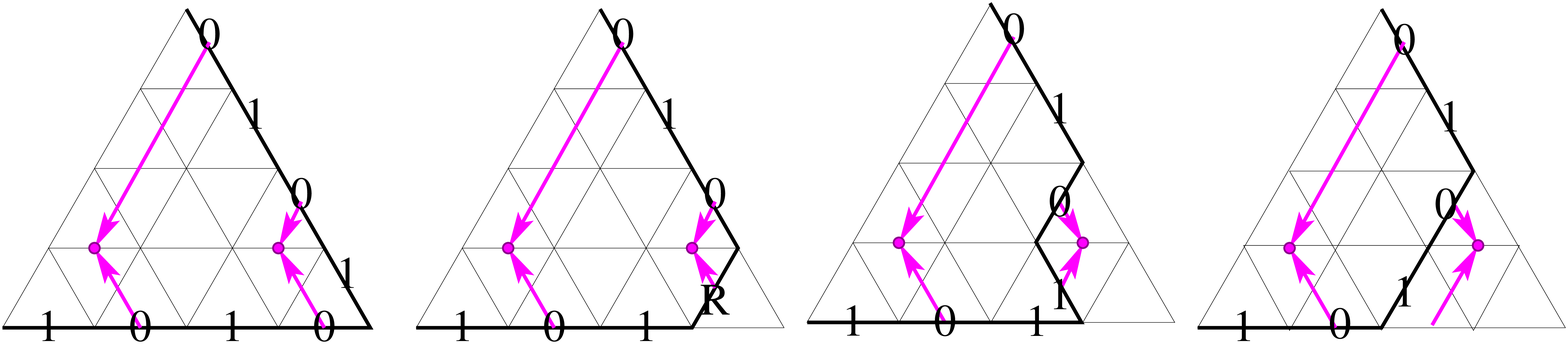,width=6.5in}}

The second shift, $2\to 4$, is nontrivial, but preserves the 
one-dimension-larger interval rank variety with essential rank conditions
$r_{12} , r_{24} \leq 1$. 
So that variety is the sweep $\sweep_{2\to 4} Y_1$; call it $Y_{1T}$.
The shift $2\to 4$ of $r_{12} , r_{34} \leq 1$ is 
$r_{12} , r_{32} \leq 1$, which is reducible; one component 
$Y_2$ is defined by $r_{13} \leq 1$ and the other, 
$Y_3$, by $r_{22} \leq 0$. 
Finally, we need to consider the intersection 
$Y_{23} := Y_2 \cap Y_3$
defined by $r_{13} \leq 1$ {\em and} $r_{22} \leq 0$. 
These Vakil varieties are associated to the following puzzle paths.
(The blue rhombus is there as reminder that we used a sweep, not a shift.)

\centerline{\epsfig{file=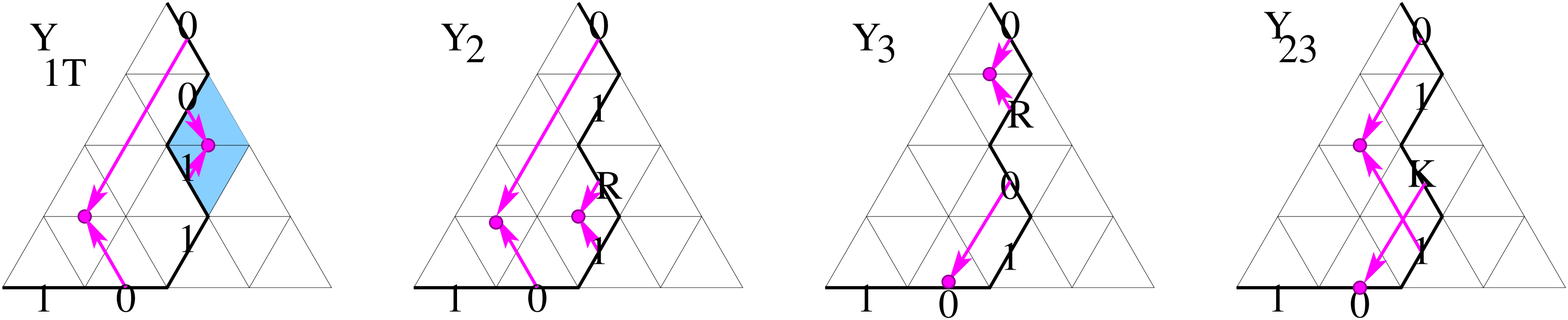,width=6.5in}}

The next three shifts, $2\to 3$, $1\to 4$, $1\to 3$ again do nothing:

\centerline{\epsfig{file=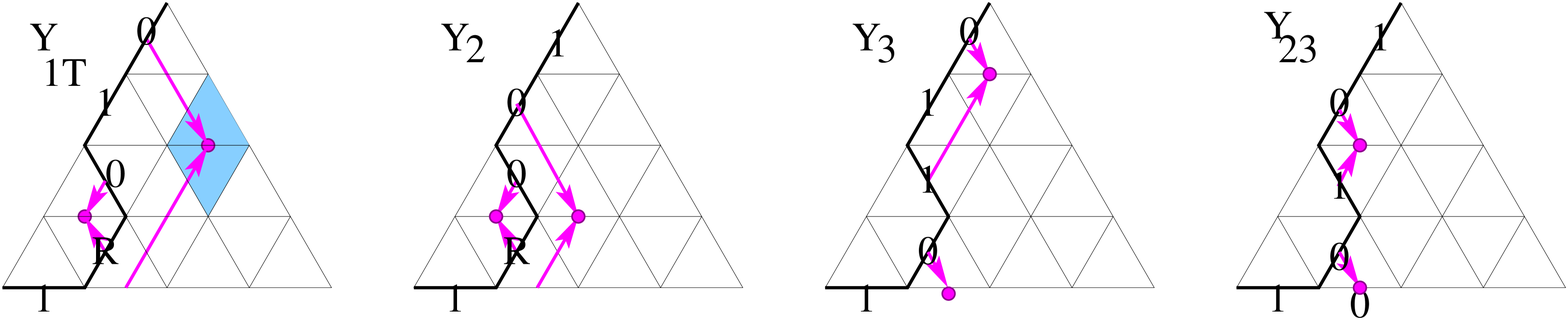,width=6.5in}}

The final shift, $1\to 2$, is only nontrivial on $Y_3$ and $Y_{23}$,
as $Y_{1T}$ and $Y_2$ are already opposite Schubert varieties.
The sweep 
$Y_{3T}$ of $Y_3$ coincides with $Y_{1T}$, and the shift 
$Y_4 := \shift_{1\to 2} Y_3$ is
the opposite Schubert variety defined by $r_{11} \leq 0$. 
The sweep 
$Y_{23T}$ of $Y_{23}$ coincides with $Y_2$,
and the shift 
$Y_5$ of $Y_{23}$ is the opposite Schubert variety
defined by $r_{11}\leq 0, r_{13}\leq 1$.

\begin{figure}[htbp]
  \centering
  \epsfig{file=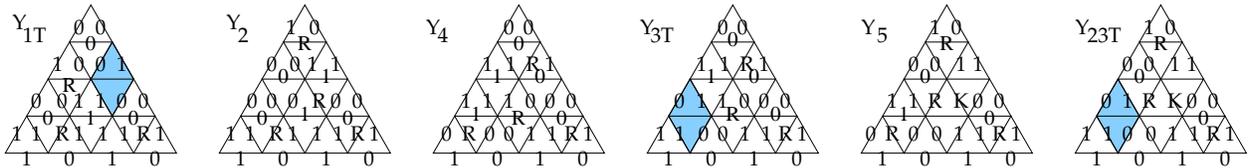,width=6.5in}
  \caption{The puzzles with $\mu=0101$, $\nu=1010$.}
  \label{fig:finalpuzex}
\end{figure}

\junk{
\begin{align}
  a_{11}& =b_{11}&  a_{12}& =b_{12}\\
  a_{21}& =b_{21}&  a_{22}& =b_{22}+c_{22}
\end{align}
}

Applying theorems \ref{thm:hpuz}-\ref{thm:ktpuz}, we have
\begin{align}
  [X_{0101}^{1010}] &= [X^{0110}] + [X^{1001}] & Y_2, Y_4 
  \tag{$H^*$}  \\
  [X_{0101}^{1010}] &= [X^{0110}] + [X^{1001}] + (y_4-y_1) [X^{1010}] 
  & Y_2, Y_4, Y_{1T}, Y_{3T} 
  \tag{$H_T^*$} \\
  [X_{0101}^{1010}] &= [X^{0110}] + [X^{1001}] - [X^{0101}]& Y_2, Y_4, Y_5
  \tag{$K$} \\
  [X_{0101}^{1010}] 
  &= \exp(y_2-y_4)\big(1 - (1-\exp(y_1-y_2)) \big) [X^{1001}] &  
  \tag{$K_T$} \\
  &+ \exp(y_2-y_4) [X^{0110}]  & \notag\\
  &+ ((1 - \exp(y_2-y_4)) - \exp(y_2-y_4)(1 - \exp(y_1-y_2)) [X^{1010}] &
  \notag \\
  &- \exp(y_2-y_4) \exp(y_1-y_2) [X^{0101}]  &\text{all terms.} \notag 
  \junk{\\
  &= \exp(y_1-y_4) [X^{1001}]  
  + \exp(y_2-y_4) ([X^{0110}] -[X^{0101}]) & \notag\\
  &+ (1 - 2\exp(y_2-y_4) + \exp(y_1-y_4) ) [X^{1010}] &
  \notag}
\end{align}

\section{Adding  a rhombus to a puzzle path}
\label{sec:filling}

The Fizzbinesque \cite{StarTrek} rules for pink rays presented in 
\S \ref{sec:variety} will be seen to interface very well with the puzzle pieces.
We deal with the easy cases first:

\begin{Lemma}\label{lem:not1then0}
  Let $\gamma$ be a non-final puzzle path,
  and call its last SE step the kink.
  (So even e.g. initial paths get an honorary kink.)
  \begin{enumerate}
  \item If the kink lies just above the bottom edge (hence the next
    step is West along the bottom), then there is a unique 
    triangle $P$ to add to $\gamma$. The resulting $\gamma'$
    has the same pink dots as $\gamma$.
  \item Otherwise the kink is followed by a step Southwest.
    If the labels on the kink and this Southwest step are {\em not} $1,0$
    respectively, then there is a unique way to add a rhombus
    (possibly consisting of two triangles) to add to $\gamma$. 
    The resulting $\gamma'$ has the same pink dots as $\gamma$.
%
  \end{enumerate}
\end{Lemma}

\begin{proof}
  There is probably no substitute for attempting the case check oneself.
  Nonetheless, we describe the results.
  \begin{itemize}
  \item If the triangle $P$ added has labels $\fslash 1,\bslash 1,\dash 1$, 
    then no pink rays move, much less any pink dots.
  \item If $P$ has labels $\fslash 0,\bslash 0,\dash 0$, 
    then there is a pink dot on its S edge, for both $\gamma$ and $\gamma'$.
  \item If $P$ has labels $\fslash 0,\bslash R,\dash 1$, 
    then there is a pink dot on its S edge, for both $\gamma$ and $\gamma'$.
  \item If $P$ has labels $\fslash R,\bslash 1,\dash 0$, 
    then it has a pink ray going NW out of $\dash 0$ in $\gamma$
    and out of $\fslash R$ in $\gamma'$.
  \end{itemize}

  For rhombi, we list the cases according to the ordered pair (the label on
  the kink, the label on the following step of $\gamma$).  The fact
  that the rhombus is unique is very easy to check.
  \begin{itemize}
  \item [($\bslash 1, \fslash 1$)]. 
    No pink rays involved in these edges at all, 
    either in $\gamma$ or $\gamma'$.
  \item [($\bslash 0, \fslash 1$)].
    In $\gamma$, there is a pink ray from $\bslash 0$, whose source
    moves one step forward in $\gamma'$.
  \item [($\bslash 1, \fslash R$)].
    In $\gamma$, there is a pink ray from $\fslash R$, whose source
    moves one step forward in $\gamma'$.
  \item [($\bslash K, \fslash 0$)].
    In $\gamma$, there is a pink ray from $\bslash K$, whose source
    moves one step forward in $\gamma'$.
  \item [($\bslash 0, \fslash 0$)], ($\bslash R, \fslash 0$), 
    ($\bslash K, \fslash 1$).
    In $\gamma$, there is a pink ray out of each edge, whose source
    moves one step forward in $\gamma'$.
  \item [($\bslash 0, \fslash R$)], ($\bslash R, \fslash 1$).
    In $\gamma$, there are pink rays out of each edge, meeting at
    a pink dot within the rhombus. 
    In $\gamma'$, the same is true, out of $\fslash 0, \bslash 1$.
  \end{itemize}
  It is straightforward to check that each resulting $\gamma'$ satisfies
  the conditions put forth at the beginning of \S \ref{sec:variety}.
\end{proof}

\begin{Lemma}\label{lem:filling}
  Let $\gamma$ be a puzzle path with a $\bslash 1$ kink,
  followed by a SW step $\fslash 0$.
  Then one can add the equivariant piece and obtain a new puzzle path $\gamma'$.

  There are two vertical rhombi made out of triangles with 
  $\bslash 1$, $\fslash 0$ on the right. At least one of those
  two rhombi can be added to $\gamma$ to obtain a new puzzle path.
  Both of those can be added iff the top $K$-piece can be added.
\end{Lemma}

\begin{proof}
  It is straightforward to check that the equivariant piece can be added,
  i.e. that the resulting $\gamma'$ satisfies 
  the conditions put forth at the beginning of \S \ref{sec:variety}.

  If the $(1,1,1)$-$\Delta$ atop the $(0,1,R)$-$\nabla$ 
  cannot be added to $\gamma$,
  it is because the first $\fslash R$ or $\dash 0$ below the kink 
  is not preceded by any $\fslash 1$.
  In this case, one cannot add the left $K$-piece.

  If the $(0,0,0)$-$\nabla$ below the $(0,1,R)$-$\Delta$ cannot be added, 
  it is because there is no $\fslash R$ below the kink, and
  the first horizontal edge is $\dash 1$, not $\dash 0$.
  In this case also, one cannot add the left $K$-piece.

  These conditions cannot hold simultaneously: the first horizontal
  edge would need to be $\dash 1$, so the first $\fslash R$ below the
  kink would need to come before any $\fslash 1$, but also there
  couldn't {\em be} any $\fslash R$ below the kink, contradiction.

  If neither condition holds, one can indeed add the left $K$-piece.
\end{proof}

Say that \defn{$\gamma$ covers $\gamma'$} if $r(\gamma)$ covers
$r(\gamma')$ in the sense of corollary \ref{cor:coveringrelations},
or equivalently, if $\barPi_{\gamma}$ is a divisor in $\barPi_{\gamma'}$.
Note that the rectangles from corollary \ref{cor:coveringrelations}
are now aligned with the puzzle columns; see the red parallelograms
in figure \ref{fig:10fillex}.

\begin{figure}[htbp]
  \centering
  \epsfig{file=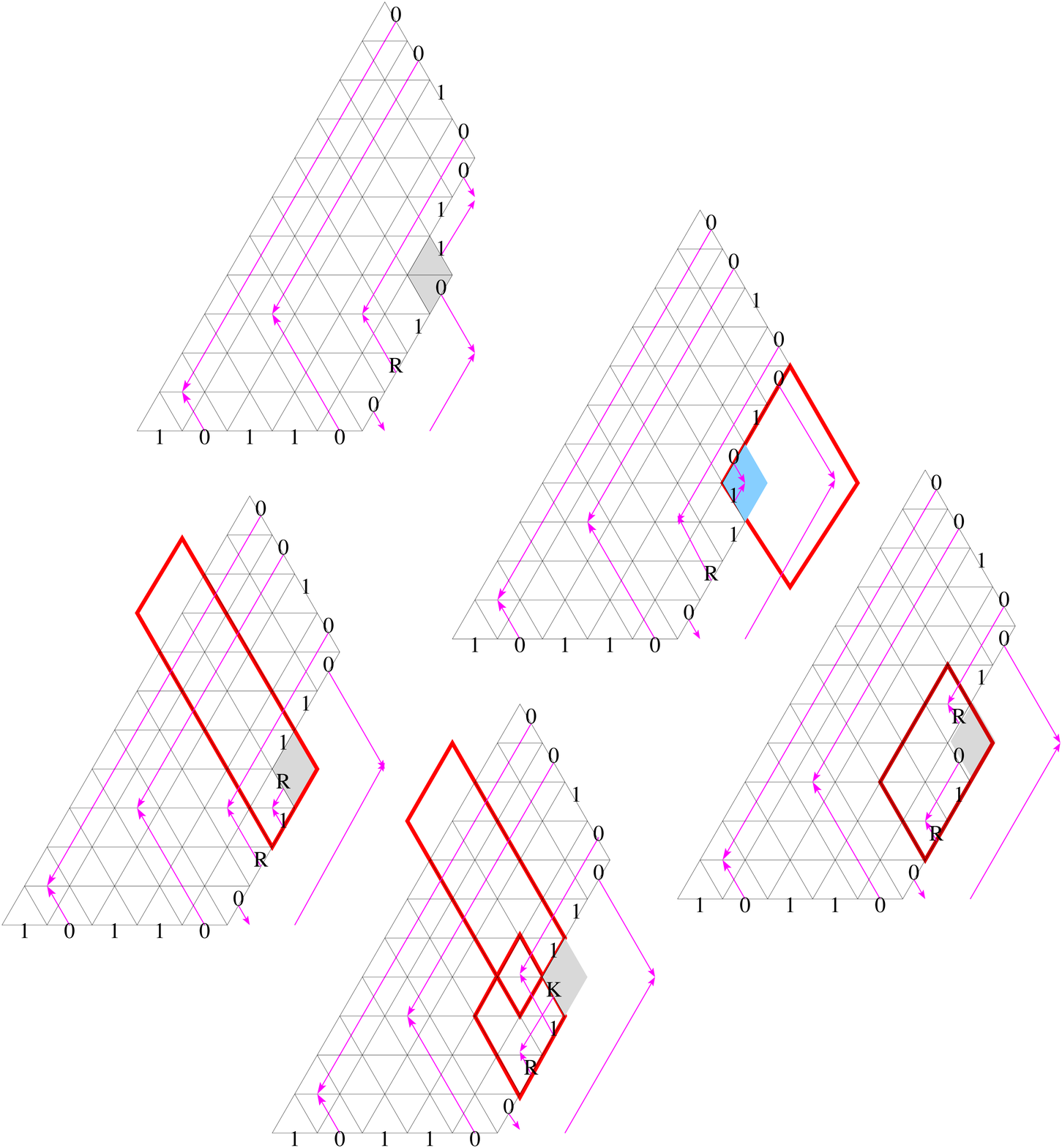,width=5in}
  \caption{A puzzle path (upper left) to which four rhombi may be added 
    (in the shaded area): the equivariant piece (top middle), two possibilities 
    of two triangles (bottom left, bottom right), and the top $K$-piece 
    (bottom middle). The red parallelograms 
    indicate the covering relations, as in corollary
    \ref{cor:coveringrelations}.}
  \label{fig:10fillex}
\end{figure}

\begin{Lemma}\label{lem:Kmeaning}
  Let $\gamma$ be a puzzle path with a $\bslash 1$ kink,
  followed by a SW step $\fslash 0$. 
  Add the equivariant piece to it giving the puzzle path $\sweep\gamma$. 

  If the $(1,0,R)$ -$\Delta$ and $(0,0,0)$ -$\nabla$ pieces
  may be added, call the resulting path $\gamma_0$; it covers $\sweep\gamma$.
  If the $(1,1,1)$ -$\Delta$ and $(1,0,R)$ -$\nabla$ pieces
  may be added, call the resulting path $\gamma_1$; it covers $\sweep\gamma$.

  Assume that one can add the top $K$-piece, producing a puzzle path
  $\gamma_K$. 
  Then $\gamma_K$ covers $\gamma_0$ and $\gamma_1$ 
  (which both exist, by lemma \ref{lem:filling}), and
  $\barPi_{\gamma_K} = \barPi_{\gamma_0} \cap \barPi_{\gamma_1}$.
  Moreover, there exist coordinate subspaces 
  $V_{\lambda_0} \in \barPi_{\gamma_0} \setminus \barPi_{\gamma_1}$,
  $V_{\lambda_1} \in \barPi_{\gamma_1} \setminus \barPi_{\gamma_0}$
  such that $\lambda_0 \not\ni i,j$ and $\lambda_1 \ni j$. 
\end{Lemma}

An example is in figure \ref{fig:10fillex}. 
Since $\barPi_{\gamma_0},\barPi_{\gamma_1}$ are irreducible of the
same dimension (codimension $1$ in $\barPi_{\sweep\gamma}$), neither one 
contains the other, and since they are defined by the vanishing of
Pl\"ucker coordinates \cite[theorem 7.4]{KLS}, the existence of
$V_{\lambda_0},V_{\lambda_1}$ is clear. 
Rather, the difficult parts of the last conclusion are the conditions on 
$\lambda_0,\lambda_1$.

\begin{proof}
  From $\sweep\gamma$ to $\gamma_1$, 
  we flip the pink ray NE from the $\bslash 1$
  in $\sweep\gamma$ to SW from the $\bslash 0$ in $\gamma_0$, and the pink dot
  in the kink NE/SW column is the only one that moves (it moves due SW, to
  the first $\fslash 1$ below the kink). 
  This is a type (2) move from corollary \ref{cor:coveringrelations}.

  From $\sweep\gamma$ to $\gamma_0$ is a type (1) move. The rhombus just filled
  is in the East corner of the relevant parallelogram, and the
  West corner is in the same NE/SW column as the last $\bslash 0$ before
  the kink, and the first NW/SE $\fslash R$ after it. 

  We leave the reader to check that when $\gamma_K$ exists, it covers both
  $\gamma_0$ and $\gamma_1$. Consequently
  $\barPi_{\gamma_K} \subseteq \barPi_{\gamma_0} \cap \barPi_{\gamma_1}$.
  To show equality, we need the stronger statement that the rank matrices
  $r({\gamma_K}) = \min(r({\gamma_0}),r({\gamma_1}))$ entrywise,
  which is also straightforward to check.

  To construct the required $\lambda_0,\lambda_1$, we construct matchings
  of their complements as in lemma \ref{lem:matchings}; each pink dot $d$
  must be matched with an edge on the bottom of puzzle in the range
  $[i(d),j(d)]$. The matchings we will use match each dot due Southwest 
  (to $i(d)$) or due Southeast (to $j(d)$). Define them by
  $$ m_0(d) = 
  \begin{cases}
    i(d) & \text{if } j(d)<j \\
    j(d) & \text{if } j(d)\geq j,
  \end{cases}
  \qquad\qquad
  m_1(d) = 
  \begin{cases}
    i(d) & \text{if } i(d)<i \\
    j(d) & \text{if } i(d)\geq i
  \end{cases}
  $$
  (where the $d$s are the pink dots of $\gamma_0$, $\gamma_1$ respectively).
  It is trivial to check that these are injective, so the complements
  $\lambda_0,\lambda_1$ of their images give the coordinates of some subspaces 
  $V_{\lambda_0} \in \barPi_{\gamma_0}$, $V_{\lambda_1} \in \barPi_{\gamma_1}$.
  To check that
  $V_{\lambda_0}, V_{\lambda_1} \notin \barPi_{\gamma_K}$, 
  we use criterion (4) of lemma \ref{lem:matchings}.
  We leave the reader to check 
  that $\lambda_0 \not\ni i,j$, $\lambda_1 \ni j$. 
\junk{big mess here
  {\em Adding the top $K$-piece.}
  To finish this case, it remains to identify the intersection 
  $\Pi_{\gamma_1} \cap \Pi_{\gamma_0}$ of the two components of 
  $\shift_{i\to j} \Pi_\gamma$. 
  Again applying lemma \ref{lem:intersectIRvars}, we know the
  intersection is a reduced union of interval positroid varieties;
  one could use general properties of flat degenerations to infer that
  it must be codimension $1$ and even pure (since $\shift_{i\to j} \Pi_\gamma$
  is reduced and hence the limit branchvariety \cite{AK}, 
  not just limit subscheme). 

  But we will be able to compute the intersection more directly.
  Let $\gamma_K$ be the result of adding the top $K$-piece to $\gamma$;
  we claim that $\Pi_{\gamma_1} \cap \Pi_{\gamma_0} = \Pi_{\gamma_K}$,
  or equivalently, that $\min(r(\gamma_1),r(\gamma_0)) = r(\gamma_K)$
  entrywise. 

  First we check that $\gamma_1$ covers $\gamma_K$ by a type (1) covering
  relation, and $\gamma_0$ covers $\gamma_K$ by a type (2).
  Call the corresponding parallelograms the \defn{$1K$-} 
  and \defn{$0K$-parallelograms}. 

  We know that $r(\gamma_1) = r(\gamma)$ away from the $1$-parallelogram
  (minus the top NW/SE and NE/SW columns), and
  $r(\gamma_0) = r(\gamma)$ away from the $0$-parallelogram
  (minus the top NW/SE and NE/SW columns).
}
\end{proof}

In particular, the union $\Pi_{\gamma_1} \cup \Pi_{\gamma_0}$ is
Cohen-Macaulay, as it is a union of two C-M schemes along a C-M divisor.
In the proof of theorem \ref{thm:filling} in the next section we will show that 
$\shift_{i\to j} \Pi_\gamma = \Pi_{\gamma_1} \cup \Pi_{\gamma_0}$.

\section{Proof of the cohomological formulae
  \ref{thm:hpuz}-\ref{thm:ktpuz}}\label{sec:ktproof}

We need a couple of lemmas about geometric shifts.

\newcommand\barB{{\overline B}}

\begin{Lemma}\label{lem:shiftbasic}
  Let $S$ be a subset of $\{1,\ldots,n\}$, and $r\in\naturals$.
  Let 
  $$ \barB_{S\leq r} := \{M \in M_{k\times n} 
  : rank(k\times |S|\text{ submatrix of $M$ with columns $S$})
  \leq r \}. $$
  Then if $i\in S$ or $j\notin S$, 
  $\shift_{i\to j} \barB_{S\leq r} = \barB_{S\leq r}$. 
  Otherwise
  $ \shift_{i\to j} \barB_{S\leq r} = \barB_{(S\setminus j \union i) \leq r}. $
\end{Lemma}

\begin{proof}
  Recall that $\shift_{i\to j} X := \lim_{t\to\infty} \exp(te_{ij})\cdot X$,
  and
  $$ \exp(te_{ij})\cdot X 
  = \{M + t(\text{column $i$ added to column $j$}) : M \in X. \} $$
  If $j\notin S$, then $rank($submatrix of $M$ with columns $S$)
  is unaffected by adding $t\cdot$column $i$ to column $j$,
  hence $\exp(te_{ij})\cdot X = X$ for all $t$.
  The same is true if $i\in S$.

  For the interesting case, we need to look closer at
  the equations defining $X$: for each $C\subseteq S$
  and $R\subseteq \{1,\ldots,k\}$, with $|C| = |R| = r+1$,
  the $M$-minor $d_{R,C}$ using rows $R$ and columns $C$ vanishes.
  Then the equations defining $\exp(te_{ij})\cdot X$ are
  $d_{R,C} = 0$ for $j\notin R$, and $d_{R,C} + t d_{R\setminus j \union i,C}=0$
  for $j\in R$. Rescaling the latter by $t^{-1}$, and taking $t\to\infty$,
  we find what are a priori {\em some} of the equations on $\shift_{i\to j} X$:
  $$ \shift_{i\to j} X \subseteq \{M \in M_{k\times n} :   
  rank(k\times |S|\text{ submatrix of $M$ with columns $S\setminus j \union i$})
  \leq r \}. $$
  Since $X$ and $\shift_{i\to j} X$ are conical affine schemes with
  the same Hilbert series (one being a degeneration of the other), 
  and this upper bound also has that same Hilbert series by $S_n$-symmetry,
  the upper bound must be tight.
\end{proof}

(A more general statement is true: if $\shift_{j\to i} X = X$, then
$\shift_{i\to j} X = (i\leftrightarrow j)\cdot X$.)
Notice that the shift $i\to j$ on columns acts backwards on these
``basic'' rank conditions; when possible, the $j\in S$ turns into an $i$.

For calculations in $H^*(\Grkn)$ or $K(\Grkn)$, we have the equation on classes
$$ [X] = [\shift_{i\to j} X] $$
but for equivariant calculations we need the following lemma.

\begin{Lemma}\label{lem:KTdegen}
  Let $X \subseteq \Grkn$ be a $T$-invariant subvariety.
  Consider the space of pairs 
  $$ F := \overline{ \{ (t, \exp(t e_{ij}) \cdot x) : t\in \AA^1, x\in X)\}
  \subseteq \AA^1 \times \Grkn } \subseteq \PP^1 \times \Grkn. $$
  Let $\pi_1,\pi_2$ denote the projections of 
  $F \subseteq \PP^1 \times \Grkn$ to $\PP^1, \Grkn$.
  Let $Y := \sweep_{i\to j} X$ be the image $\pi_2(F)$.

  If the map $F \to Y$ has degree $d$ (taken to be $0$ if the fibers
  are $\PP^1$s), then
  we have the following equality between $H^*_T(\Grkn)$-classes:
  $$ [X] = d (y_j-y_i) [Y] + [\shift_{i\to j} X]. $$

  Identify $\Grkn^T$, the set of coordinate $k$-planes,
  with the collection of $k$-element subsets of $\{1,\ldots,n\}$.
  If $\lambda \in X^T$ is a point such that $i\in \lambda, j\notin \lambda$,
  and $(i\leftrightarrow j)\cdot\lambda \notin X^T$,
  then $\dim Y = \dim X + 1$ and $d=1$.

  Now assume that $Y := \sweep_{i\to j} X$ has rational singularities,
  and that $d$ is indeed $1$. Then in $K_T(\Grkn)$ one has
  $$ [X] = \big(1-\exp(y_i-y_j)\big) [Y] + \exp(y_i-y_j) [\shift_{i\to j} X].$$
\end{Lemma}

\begin{proof}
  If we let $T$ act on this $\PP^1$ by 
  $diag(t_1,\ldots,t_n)\cdot z := t_i t_j^{-1} z$, 
  and hence let $T$ act on $\PP^1 \times \Grkn$ diagonally, 
  then $F$ is $T$-invariant.

  The $H^*_T$ and $K_T$ calculations are very similar, so we do the
  harder one, $K_T$.
  We must be careful to distinguish between $K_T$ (cohomology) and
  $K^T$ (homology) in the following, because $F$ is unlikely to be smooth.
  Take the equation in $K_T(\PP^1)$
  $$ 1 - \exp(y_i-y_j) = [\{0\}] - \exp(y_i-y_j) [\{\infty\}], $$
  pull back with $\pi_1^*$ to $K_T(F)$, and cap with the fundamental class
  to get an equation in $K^T(F)$:
  $$ (1 - \exp(y_i-y_j)) [F] 
  = [\{0\} \times X] - \exp(y_i-y_j) [\{\infty\} \times \shift_{i\to j} X] $$
  Push forward with $(\pi_2)_*$ 
  to get an equation in $K^T(\Grkn)$:
  $$ (1 - \exp(y_i-y_j)) (\pi_2)_* [F] 
  = [X] - \exp(y_i-y_j) [\shift_{i\to j} X]. $$
  If $\shift_{i\to j} X = X$, then $F = \PP^1 \times X$, and 
  $(\pi_2)_* [F] = [Y] = [X]$, and this is trivial. 
  So assume not. Then $\dim F = \dim Y = 1 + \dim X$. 

  Since the map $\pi_2: F \onto Y$ is birational (by $d=1$),
  and since $Y$ was assumed to have rational singularities, 
  $(\pi_2)_* [F] = [Y]$. Finally we use smoothness of $\Grkn$ to
  move the equation from $K^T(\Grkn)$ to $K_T(\Grkn)$.

  The derivation for $H^*_T$ proceeds from the $H^*_T(\PP^1)$-equation
  $$ y_j-y_i = [\{0\}] - [\{\infty\}] $$
  and the calculation $(\pi_2)_* [F] = d [Y]$.

  It remains to prove the second claim, that $\dim Y = \dim X + 1$ and $d=1$.
  Since $\shift_{i\to j} X \supseteq \shift_{i\to j} \lambda 
  = (i\leftrightarrow j)\cdot\lambda \notin X$, the (irreducible) image $Y$ 
  of $F$ contains $X$ strictly, so is of larger dimension. 
  But $\dim F = \dim X + 1$, so its image $Y$ can only be one
  dimension larger.

  To show that $d=1$, we show that the preimage in $F$ of 
  $(i\leftrightarrow j)\cdot \lambda \in Y$ is the single, reduced point
  $(\infty, \lambda)$. 
  The two extreme cases of $X$ will turn out to be $X_\leq := \{\lambda\}$ 
  and $X_\geq := \{ V : p_{(i\leftrightarrow j)\cdot\lambda}(V)=0 \}$,
  where $p_\lambda$ is the corresponding Pl\"ucker coordinate.
  By assumption, $X_\leq \subseteq X$. 
  The open set $p_{(i\leftrightarrow j)\cdot\lambda}(V)=0$ in $\Grkn$ is also
  the open Bia\l ynicki-Birula stratum for a one-parameter subtorus of $T$, 
  with $(i\leftrightarrow j)\cdot\lambda$ the attractive fixed point.
  By the $T$-invariance of $X$, 
  if $X$ were to intersect this open set, it would contain 
  $(i\leftrightarrow j)\cdot\lambda$. Since $X$ is reduced, it must lie in 
  the divisor complementary to this open set.

  Hence we can trap $X$ in $X_\leq \subseteq X \subseteq X_\geq$. 
  Let $F_\leq,F_\geq$ be the corresponding families, so we can similarly
  trap the fiber $  F_{(i\leftrightarrow j)\cdot \lambda}$ over
  $(i\leftrightarrow j)\cdot \lambda$ in
  $$ (F_\leq)_{(i\leftrightarrow j)\cdot \lambda} \subseteq
  F_{(i\leftrightarrow j)\cdot \lambda} \subseteq
  (F_\geq)_{(i\leftrightarrow j)\cdot \lambda}. $$
  The lower bound contains 
  the point $(\infty,(i\leftrightarrow j)\cdot\lambda)$.
  The large family $F_\geq$ is defined by the equation
  $$ F_\geq = \{ ([a,b], V) : 
  a p_\lambda(V) + b p_{(i\leftrightarrow j)\cdot\lambda}(V) = 0 \} $$ 
  (a particular case of the calculation in the proof 
  of lemma \ref{lem:shiftbasic}). Over the point
  $(i\leftrightarrow j)\cdot \lambda$, this equation is $b=0$, 
  defining the same point $(\infty,(i\leftrightarrow j)\cdot\lambda)$.
\end{proof}

\junk{
Theorem \ref{thm:rankmatrices} associated a matrix Schubert variety
$\barX_\pi \subseteq M_{(k+n)\times(n+k)}$ to any interval rank matrix $r$.
Recall from \cite[theorem A]{GrobnerGeom} that the corresponding class
in $K_{T^{k+n}\times T^{n+k}}(M_{(k+n)\times(n+k)}) 
\iso \integers[\exp(\pm z_1),\ldots,\exp(\pm z_{n+k}),
\exp(\pm y_1),\ldots,\exp(\pm y_{n+k})]$ is the
\defn{double Grothendieck polynomial} $G_\pi$. 
These satisfy the following \defn{Monk formula:}

\begin{Theorem*}\cite[Corollary 8.2]{LP}
\end{Theorem*}
}

\begin{Theorem}\label{thm:filling}
  Let $\gamma$ be a non-final puzzle path, whose kink is followed by 
  a SW edge. Let $p_1,\ldots,p_d$ be the rhombi that can be added to $\gamma$,
  giving the puzzle paths $\gamma'_1,\ldots,\gamma'_d$. 
  Then each $\Pi_{\gamma'_i}$ is $\shift_{i(p)\to j(p)}$-invariant. 
  Let $\Phi(E^*,\rho) \in E^*(pt)$ be the factors associated to the 
  cohomology theory $E^* \in \{H^*, H^*_T, K, K_T\}$, 
  as defined in \S \ref{ssec:positivity}.

  Then we have the following equality between classes in $E^*(\Grkn)$:
  $$ [\Pi_\gamma] = \sum_{i=1}^d \Phi(E^*,p_i)\ [\Pi_{\gamma'_i}] $$
\end{Theorem}

\begin{proof}
  If the labels on the kink-then-SW-edge are not $\bslash 1,\fslash 0$,
  then lemma \ref{lem:not1then0} applies: there is a unique $p = p_1$,
  with $\Phi(K_T,p_i) = 1$, and $\Pi_{\gamma} = \Pi_{\gamma'_1}$.
  The equation on $E^*(\Grkn)$ classes is then trivial.
  So we assume hereafter that we are in the interesting case, that
  the labels on the kink-then-SW-edge are indeed $\bslash 1,\fslash 0$.

  \junk{
  The fact that $\gamma$ has no essential edges due SE of $p$ says
  that each essential rank condition defining $\Pi_\gamma$ is
  $\shift_{i(p)\to j(p)}$-invariant. 
  Hence $\Pi_{\gamma'_1}$ is $\shift_{i(p)\to j(p)}$-invariant. 
  }

  \newcommand\sidein{\rotatebox{90}{$\in$}}

  \junk{  
    At this point it will be more convenient to assume that $E^*$ {\em is}
    equivariant, so $H^*_T$ or $K_T$, and simply observe that the
    nonequivariant formul\ae\ follow from the equivariant under the
    specialization $\yy \mapsto 0, e^\yy \mapsto 1$.
    The two cases $H^*_T$ and $K_T$ are very similar, so we deal with
    the harder one, $K_T$. As we will make use of other torus actions
    in the proof we will refer to this one more specifically as $K_{T^n}$.

    Now we consider the interesting case, 
    that the labels on the kink-then-SW-edge are $\bslash 1,\fslash 0$.
    To compute classes in $K_{T^n}(\Grkn)$, we use the following maps of spaces
    $$ \Grkn \stackrel{\dom GL(k)}{\fromonto} \St_{k,n} 
    \into C \into M_{(k+n)\times (n+k)} = M_{(k+n)\times (n+k)} $$
    where $T^n$ acts in the standard way on $\Grkn$ and by right multiplication 
    on the space $C$ from the proof of theorem \ref{thm:rankmatrices},
    on $\St_{k,n}$ which we identify with a subset of $C$.
    For the inclusion $\St_{k,n} \to C$ to be $GL(k)\times T^n$-equivariant,
    we have $GL(k)$ act on the left of $C$, performing row operations on
    the first $k$ rows, and the $T^n$ acting on both sides, scaling the
    first $n$ columns and unscaling the last $n$ rows, in order to preserve 
    the $n\times n$ identity matrix in $C$'s elements.

    These maps induce the following maps on cohomology groups, and
    elements thereof:
    $$\begin{array}{ccccccccc}
      K_{T^n}(\Grkn) 
      &\iso& K_{GL(k)\times T^n}(\St_{k,n}) 
      &\from& K_{GL(k)\times T^n}(C) 
      &\from& K_{GL(k)\times T^n}(M_{(k+n)\times (n+k)}) 
      \\
      \sidein && \sidein && \sidein && \sidein 
      \\
      [0pt] [\Pi_r] 
      &\leftrightarrow& [\St_{k,n} \cap \barX_\pi]
      &\mapsfrom& [C \cap \barX_\pi]
      &\mapsfrom& [\barX_\pi]
    \end{array}$$
    Here $\pi \in S_{k+n}$ is the permutation defined from $r$ in the proof of
    theorem \ref{thm:rankmatrices}.
    To follow the elements, we use the fact that $[A][B] = [A\cap B]$ if
    $A,B$ are both Cohen-Macaulay and $A\cap B$ has the expected dimension.

    We wish to compute in $K_{GL(k)\times T^n}(M_{(k+n)\times (n+k)})$.
    To do so, we use
    $$ K_{GL(k)\times T^n}(M_{(k+n)\times (n+k)}) 
    \fromonto K_{(GL(k)\times T^n)\times(T^n\times T^k)}(M_{(k+n)\times (n+k)}) 
    \into K_{(T^k\times T^n) \times (T^n\times T^k)}(M_{(k+n)\times (n+k)}) $$
    where both maps come from restricting the action.
    In this last ring, the class $[\barX_\pi]$ is given by the
    double Grothendieck polynomial $G_\pi$ \cite[theorem A]{GrobnerGeom}, 
    and the desired formula will follow from the transition formula for
    double Grothendieck polynomials \cite{Lascoux}.
  }

  We have already analyzed the interesting case on the puzzle side,
  in lemma \ref{lem:filling}. 
  To replicate the $2$ or $4$ fillings that show up there, we will apply
  lemma \ref{lem:KTdegen} to $\Pi_\gamma$.
  To analyze the sweep and shift of $\Pi_\gamma$, we will need the
  puzzle-theoretic version of corollary \ref{cor:coveringrelations},
  in which the rectangles in the partial permutation matrix $J(r)$ are
  replaced by parallelograms in the puzzle (edges parallel to the rhombi). 
  Type (2) covering relations from the corollary now correspond to
  pink dots moving SW or SE.

  Let $\sweep \gamma$ denote the path constructed by adding the
  equivariant rhombus to $\gamma$. Let its NE/SW and NW/SE columns be $(i,j)$.
  Let $r(\gamma),r(\sweep \gamma)$ be the interval rank matrices associated
  to these two puzzle paths as in \S \ref{ssec:vakilvariety}.

  First claim: $\Pi_{\sweep \gamma} \supseteq \Pi_\gamma$. 
  Indeed, transferring the Bruhat order 
  from corollary \ref{cor:coveringrelations} on interval rank matrices
  over to the set of Vakil varieties, we see $\Pi_{\gamma}$ 
  covers $\Pi_{\sweep \gamma}$; the two pink dots in NE/SW column $i,i+1$
  exchange their NW/SE columns. 
  So $\Pi_\gamma$ is codimension $1$ in $\Pi_{\sweep \gamma}$.

  Second claim: 
  there is a $T$-fixed point $V_\lambda \in \Pi_\gamma$ such that
  $j \notin \lambda$, $(i\leftrightarrow j)\cdot V_\lambda \notin \Pi_\gamma$.
  Let $d$ be the Southern of the two pink dots of $\gamma$ that move 
  for $\sweep\gamma$, and define $m' : supp(J(\gamma)) \to \{1,\ldots,n\}$ by
  $$ m'(e) =
  \begin{cases}
    i(e) & \text{if } j(e)<j(d) \\
    j(e) & \text{if } j(e)\geq j(d). 
  \end{cases}
  $$
  By lemma \ref{lem:matchings} (2), the complement $\lambda$ of 
  the image of $m'$ has $V_\lambda \in \Pi_\gamma$. Since $m'(d) = j(d) = j$, 
  we have $j\notin \lambda$. 
  Now define $m : supp(J(\sweep\gamma)) \to \{1,\ldots,n\}$ by
  $$ m(e) =
  \begin{cases}
    i(e) & \text{if } j(e)\leq j(d) \\
    j(e) & \text{if } j(e)>j(d).
  \end{cases}
  $$
  Then $image(m) = (i\leftrightarrow j)\cdot image(m')$, and $m$ satifies
  lemma \ref{lem:matchings} (4), 
  so $(i\leftrightarrow j)\cdot V_\lambda 
  \in \Pi_{\sweep\gamma} \setminus \Pi_\gamma$. 

  Third claim: 
  $\Pi_{\sweep\gamma} = \sweep_{i\to j} \Pi_\gamma$.
  By the first claim,
  $\sweep_{i\to j} \Pi_\gamma \subseteq \sweep_{i\to j} \Pi_{\sweep\gamma}$.
  By lemma \ref{lem:filling}, $\sweep\gamma$ has no essential rank
  conditions on intervals $[k,l]$ with $i < k \leq j \leq l$,
  so $\Pi_{\sweep\gamma} = \shift_{i\to j} \Pi_{\sweep\gamma}
  = \sweep_{i\to j} \Pi_{\sweep\gamma}$. 
  Together, $\sweep_{i\to j} \Pi_\gamma \subseteq \Pi_{\sweep\gamma}$.
  By the second claim, $\Pi_\gamma$ is not $\shift_{i\to j}$-invariant, 
  hence $\dim \sweep_{i\to j} \Pi_\gamma = \dim \Pi_\gamma + 1 
  = \dim \Pi_{\sweep\gamma}$. 
  Thus the containment of varieties is an equality.

  Since $\Pi_{\sweep\gamma}$ is a positroid variety, it has
  rational singularities (\cite[corollary 7.10]{KLS}, or 
  use theorem \ref{thm:rankmatrices} and the corresponding fact 
  about Kazhdan-Lusztig varieties). The second claim allows us to
  apply lemma \ref{lem:KTdegen} to $\Pi_\gamma$, obtaining
  $$ [\Pi_\gamma] 
  = \big(1-\exp(y_i-y_j)\big) [\Pi_{\sweep\gamma}]
  + \exp(y_i-y_j) [\shift_{i\to j} \Pi_\gamma]
  \qquad \in K_T(\Grkn)$$
  and
  $$ [\Pi_\gamma] 
  = (y_j-y_i) [\Pi_{\sweep\gamma}]
  + [\shift_{i\to j} \Pi_\gamma]
  \qquad \in H^*_T(\Grkn). $$
  These are not quite of the form required in the theorem statement, 
  as $\shift_{i\to j} \Pi_\gamma$ is not necessarily irreducible
  (though, being a flat degeneration of $\Pi_\gamma$, it is 
  necessarily equidimensional). 
  The remainder of the proof is the analysis of $\shift_{i\to j} \Pi_\gamma$.

  Fourth claim: 
  $r(\gamma)_{i+1,j} < r(\sweep\gamma)_{i,j-1}$
  (see the top two pictures in figure \ref{fig:10fillex} for an example).
  First observe $r(\gamma)_{i+1,j} = r(\gamma)_{ij} - 1$ 
  and $r(\sweep\gamma)_{i,j-1} = r(\sweep\gamma)_{ij}$, 
  so it is enough to show $r(\gamma)_{ij} \leq r(\sweep\gamma)_{ij}$.
  The pink dots of $\gamma$ and of $\sweep\gamma$ agree except for two arranged
  roughly east/west in $\gamma$ that move roughly north/south in $\sweep\gamma$,
  and give the same count at position $(i,j)$, so
  $r(\gamma)_{ij} \leq r(\sweep\gamma)_{ij}$.

  Now we apply lemma \ref{lem:shiftbasic}:
  since $\Pi_\gamma \subseteq B_{[i+1,j] \leq r(\gamma)_{i+1,j}}$, we know 
  $\shift_{i\to j} \Pi_\gamma \subseteq B_{[i,j-1] \leq r(\gamma)_{i+1,j}}$.
  Hence
  $$ \shift_{i\to j} \Pi_\gamma 
  \subseteq \Pi_{\sweep\gamma} \cap B_{[i,j-1] \leq r(\gamma)_{i+1,j}} $$
  and by the fourth claim, the right-hand side is properly contained
  in $\Pi_{\sweep\gamma}$. Since
  $\dim\shift_{i\to j} \Pi_\gamma = \dim\Pi_\gamma = \dim\Pi_{\sweep\gamma}-1$
  as in the third claim, the equidimensional left side 
  $\shift_{i\to j} \Pi_\gamma$ must consist of some of the geometric
  components of the right-hand side.
  (The containment will turn out to be an equality.)

  Since both $\Pi_{\sweep\gamma}$ and $B_{[i,j-1] \leq r(\gamma)_{i+1,j}}$
  are interval positroid varieties, by lemma \ref{lem:intersectIRvars}
  their intersection is a reduced union of interval positroid varieties.
  Since the dimension count above indicates that the intersection 
  is codimension $1$ in $\Pi_{\sweep\gamma}$, 
  we need to look for interval rank matrices covered by $r(\sweep\gamma)$
  in the covering relations from corollary \ref{cor:coveringrelations},
  and they must lower the rank bound on columns $[i,j-1]$.
  The reader may wish to study figure \ref{fig:10fillex} while
  following the next argument.

  {\em Adding the $(1,0,R)-\Delta$ and the $(0,0,0)-\nabla$. }
  We first consider the covering relations of type (1), coming from
  a parallelogram in $\sweep\gamma$'s puzzle with pink dots only in the
  left and right corners. 
  Moreover, the $(i,j-1)$ rhombus should be in the parallelogram but not
  on its top two edges. That forces the pink dot at position $(i,j)$
  in $\sweep\gamma$'s puzzle to be in the rhombus, and moreover 
  to be the pink dot in the right-hand corner.

  Which pink dot could be in the left-hand corner? Being left of the
  $(i,j)$ rhombus, it must be the intersection of a NW- and a SW-pointing ray,
  along the SW and NW sides of the parallelogram. If we order those dots
  according to their NW/SE column, the parallelogram with right-hand corner
  $(i,j)$ and left corner the $p$th dot will contain in its interior
  the $q$th dot for each $q>p$. Since we only want parallelograms with
  no pink dots in the interior, we must take the last dot in this order, 
  NW of the first $\fslash R$ or $\dash 0$ below the kink.
  Call the resulting parallelogram the \defn{$0$-parallelogram}
  for later reference.
  
  If we add the $(1,0,R)-\Delta$ and the $(0,0,0)-\nabla$ triangular pieces
  to $\gamma$, we get another puzzle path $\gamma_1$ whose pink dots match 
  those of $\sweep\gamma$, except that the pink dots in the left and
  right of the $0$-parallelogram have moved to the top and bottom.

  To sum up: there is at most one relevant covering relation of type (1),
  and it is effected exactly by adding
  the $(1,0,R)-\Delta$ and the $(0,0,0)-\nabla$ triangular pieces to $\gamma$.
  (If there is no $\fslash R$ or $\dash 0$ below the kink, 
  then adding the $(1,0,R)-\Delta$ and the $(0,0,0)-\nabla$ triangular pieces 
  to $\gamma$ produces an illegal puzzle path.) 

  {\em Adding the $(1,1,1)-\Delta$ and the $(1,0,R)-\nabla$. }
  Now we consider the covering relations of type (2), which are most
  easily thought about by adding pink dots just outside the puzzle triangle
  on the NW and NE sides, in each column (NW/SE or NE/SW) that doesn't 
  already have a pink dot. We again want a parallelogram in $\sweep\gamma$'s 
  puzzle (now allowed to reach slightly outside) 
  with pink dots only in the left and right corners, such that
  the $(i,j-1)$ rhombus is in the parallelogram but not
  on its top two edges. That again forces the right-hand corner
  to be at position $(i,j)$. The left corner contains the pink dot
  at position $(0,j')$ with the maximum $j'>j$, 
  i.e. NW of the first $\fslash 1$ below the kink.
  Call the resulting parallelogram the \defn{$1$-parallelogram}
  for later reference.

  If we add the $(1,1,1)-\Delta$ and the $(1,0,R)-\nabla$ triangular pieces
  to $\gamma$, we get another puzzle path $\gamma_0$ whose pink dots match 
  those of $\sweep\gamma$, except that the pink dots in the left and
  right of the $1$-parallelogram have moved to the top and bottom. 
  Inside the triangle, (only) one pink dot has moved SW from $(i,j)$
  to be just NW of the first $\fslash 1$ below the kink.

  So far we have analyzed the upper bound 
  $\Pi_{\sweep\gamma} \cap B_{[i,j-1] \leq r(\gamma)_{i+1,j}}$; 
  it is reduced, and has at most the components $\Pi_{\gamma 0},\Pi_{\gamma 1}$
  predicted by adding
  two triangles to $\gamma$ (in particular, at most two components).
  Since it contains $\shift_{i\to j} \Pi_\gamma$, whose dimension 
  is $\dim \Pi_{\sweep\gamma} - 1$, the upper bound must have at least
  one of those two possible components.

  If $\Pi_{\sweep\gamma} \cap B_{[i,j-1] \leq r(\gamma)_{i+1,j}}$ has
  only one component, say $\Pi_{\gamma_0}$, then 
  $$ [\Pi_\gamma] 
  = \big(1-\exp(y_i-y_j)\big) [\Pi_{\sweep\gamma}]
  + \exp(y_i-y_j) [\Pi_{\gamma_0}], $$
  which was to be proved. The $\shift_{i\to j} \Pi_\gamma = \Pi_{\gamma_1}$
  case is exactly the same.

  The remaining case is that the upper bound has two components
  $\Pi_{\gamma_1} \cup \Pi_{\gamma_0}$
  (i.e. both ways of adding two triangles to $\gamma$ result in 
  valid puzzle paths); we need to show that $\shift_{i\to j} \Pi_\gamma$
  contains each entire component. Since $\shift_{i\to j} \Pi_\gamma$ 
  is (set-theoretically) equidimensional, it is enough to show it
  contains a point in each of $\Pi_{\gamma_1} \setminus \Pi_{\gamma_0}$,
  $\Pi_{\gamma_0} \setminus \Pi_{\gamma_1}$. 
  We did exactly this at the end of lemma \ref{lem:Kmeaning}; 
  the conditions given there on $\lambda_0,\lambda_1$ ensure that 
  they lie in $\shift_{i\to j} \Pi_\gamma$.
  Hence
  $$ \shift_{i\to j} \Pi_\gamma = \Pi_{\gamma_1} \cup \Pi_{\gamma_0} $$
  and
  $$ [\shift_{i\to j} \Pi_\gamma] 
  = [\Pi_{\gamma_1}] + [\Pi_{\gamma_0}] - [\Pi_{\gamma_1} \cap \Pi_{\gamma_0}] 
  = [\Pi_{\gamma_1}] + [\Pi_{\gamma_0}] - [\Pi_{\gamma_K}] 
  $$
  the latter by lemma \ref{lem:Kmeaning}. Consequently
  $$ [\Pi_\gamma] 
  = \big(1-\exp(y_i-y_j)\big) [\Pi_{\sweep\gamma}]
  + \exp(y_i-y_j) 
  ([\Pi_{\gamma_1}] + [\Pi_{\gamma_0}] - [\Pi_{\gamma_K}])
  $$
  which was to be proved.
\end{proof}

\begin{proof}[Proof of theorems \ref{thm:hpuz}-\ref{thm:ktpuz}]
  Let $\gamma$ be the initial puzzle path having $\mu$ on the NE side
  and $\nu$ on the S side, both read left-to-right. Add rhombi to it
  in the filling order from figure \ref{fig:degenorder} (and add
  triangles at the end of each NE/SW column). Along the way, there may
  be choices (exactly when the kink and next edge are $\bslash 1,\fslash 0$);
  a record of the choices made is exactly a puzzle $P$. 
  At the end, we have a final puzzle path $\lambda(P)$, 
  whose labels are determined by the NW side of $P$.

  So iterating theorem \ref{thm:filling} $n\choose 2$ times, we obtain
  $$ [\Pi_\gamma] 
  = \sum_{\text{puzzles $P$ with $\mu$ on NE, $\nu$ on S}} \quad
  \left( \prod_{\rho\in P} \Phi(E^*,\rho) \right)
  [\Pi_{\lambda(P)}] $$
  By propositions \ref{prop:initialpath} and \ref{prop:finalpath},
  $\Pi_\gamma = X^\nu_\mu$ and $\Pi_{\gamma'(P)} = X^{\text{NW side of $P$}}$.
  So
  $$ [X^\nu_\mu] 
  = \sum_\lambda 
  \sum_{\text{puzzles $P$ with $\mu$ on NE, $\nu$ on S, $\nu$ on NW}} \quad
  \left( \prod_{\rho\in P} \Phi(E^*,\rho) \right)
  [X^\lambda]. $$
\end{proof}



\end{document}